\documentclass[onefignum,onetabnum]{siamonline171218}

\usepackage{lipsum}
\usepackage{amsfonts}
\usepackage{graphicx}
\usepackage{epstopdf}
\usepackage{algorithmic}
\ifpdf
  \DeclareGraphicsExtensions{.eps,.pdf,.png,.jpg}
\else
  \DeclareGraphicsExtensions{.eps}
\fi

\usepackage{enumitem}
\setlist[enumerate]{leftmargin=.5in}
\setlist[itemize]{leftmargin=.5in}

\newenvironment{gh}{\color{black}} 

\newenvironment{revision}{\color{black}} 

\newsiamremark{remark}{Remark}
\newsiamremark{hypothesis}{Hypothesis}
\crefname{hypothesis}{Hypothesis}{Hypotheses}
\newsiamthm{claim}{Claim}

\headers{DSOS and SDSOS Optimization}{A. A. Ahmadi and A. Majumdar}

\title{DSOS and SDSOS Optimization:\\ More Tractable Alternatives to\\ Sum of Squares and Semidefinite Optimization\thanks{An extended abstract for this work has appeared in~\cite{dsos_ciss}.
\funding{ Amir Ali Ahmadi was partially supported by the DARPA Young Faculty Award, the Sloan Fellowship, the NSF CAREER Award, the AFOSR Young Investigator Program Award, and the Google Research Award.}}}

\author{
  Amir Ali Ahmadi\thanks{Department of Operations Research and Financial Engineering, Princeton University.
    (\email{a\_a\_a@princeton.edu}, \url{http://aaa.princeton.edu/}).}
  \and
  Anirudha Majumdar  \thanks{  Department of Mechanical and Aerospace Engineering, Princeton University. (\email{ani.majumdar@princeton.edu}, \url{https://irom-lab.princeton.edu/majumdar/}).}
  }
  
\usepackage{amsopn}

\ifpdf
\hypersetup{
  pdftitle={DSOS and SDSOS Optimization:\\ More Tractable Alternatives to\\ Sum of Squares and Semidefinite Optimization},
    pdfauthor={Amir Ali Ahmadi and Anirudha Majumdar}}
\fi

\usepackage{amsmath} 
\usepackage{amssymb}  
\usepackage{hyperref} 
\usepackage{color}
\usepackage{graphicx}        
\usepackage[hang,tight]{subfigure}
\usepackage{stackengine}
\usepackage{graphicx}
\usepackage{color}

\usepackage{subfigure}
\usepackage{algorithm}
\newsiamthm{Example}{Example}

\def\ani#1{{\color{black}#1}}

\def\aaa#1{{\color{black}#1}}

\begin{document}

\maketitle

\begin{abstract} In recent years, optimization theory has been greatly impacted by the advent of sum of squares (SOS) optimization. The reliance of this technique on large-scale semidefinite programs however, has limited the scale of problems to which it can be applied.
  %
\aaa{In this paper, we introduce DSOS and SDSOS optimization as linear programming and second-order cone programming-based alternatives to sum of squares optimization that allow one to trade off computation time with solution quality.} These are optimization problems over certain subsets of sum of squares polynomials (or equivalently subsets of positive semidefinite matrices), which can be of interest in general applications of semidefinite programming where scalability is a limitation. We show that some basic theorems from SOS optimization which rely on results from real algebraic geometry are still valid for DSOS and SDSOS optimization. Furthermore, we show with numerical experiments from diverse application areas---polynomial optimization, statistics and machine learning, derivative pricing, and control theory---that with reasonable tradeoffs in accuracy, we can handle problems at scales that are currently significantly beyond the reach of traditional sum of squares approaches. Finally, we provide a review of recent techniques that bridge the gap between our DSOS/SDSOS approach and the SOS approach at the expense of additional running time.
  The \revision{Supplementary Material} of the paper introduces an accompanying MATLAB package for DSOS and SDSOS optimization.
  


  
  
  %
  %
  %
  %
  %
  %
  %
  %
  %
  %
\end{abstract}

\begin{keywords}
 Sum of squares optimization, polynomial optimization, nonnegative polynomials, semidefinite programming, linear programming, second order cone programming.
\end{keywords}

\begin{AMS}
   65K05, 90C25, 90C22, 90C05, 90C90, 12Y05, 93C85, 90C27
\end{AMS}

\section{Introduction}
For which values of the real coefficients $c_1,c_2,c_3,$ is the polynomial 
\begin{equation}\label{eq:opening.poly.ex}
p(x_1,x_2)= c_1x_1^4-6x_1^3x_2-4x_1^3+c_2x_1^2x_2^2+10x_1^2+12x_1x_2^2+c_3x_2^4
\end{equation}
\emph{nonnegative}, i.e., satisfies $p(x_1,x_2)\geq 0$ for all $(x_1,x_2)\in\mathbb{R}^2$? The problem of optimizing over nonnegative polynomials---of which our opening question is a toy example---is fundamental to many problems of applied and computational modern mathematics. In such an optimization problem, one would like to impose constraints on the coefficients of an unknown polynomial so as to make it nonnegative, either globally in $\mathbb{R}^n$ (as in the above example), or on a certain \emph{basic semialgebraic set}, i.e., a subset of $\mathbb{R}^n$ defined by a finite number of polynomial inequalities. We will demonstrate shortly (see Section~\ref{sec:intro_applications}) why optimization problems of this kind are ubiquitous in applications and universal in the study of questions dealing in one way or another with polynomial equations and inequalities.

Closely related to nonnegative polynomials are polynomials that are sums of squares. We say that a polynomial $p$ is a \emph{sum of squares} (sos) if it can be written as $p=\sum_i q_i^2$ for some (finite number of) polynomials $q_i$. For example, the polynomial in (\ref{eq:opening.poly.ex}) with $c_1=13,c_2=1,c_3=4$ is sos since it admits the decomposition $$p(x_1,x_2)=(x_1-2x_1^2)^2+(3x_1+2x_2^2)^2+(x_1x_2-3x_1^2)^2.$$

The relationship between nonnegative and sum of squares polynomials has been a classical subject of study in real algebraic geometry. For example, a result of Hilbert from 1888~\cite{Hilbert_1888} states that all nonnegative bivariate polynomials of degree four are sums of squares. It follows, as a special case, that the sets of coefficients for which the polynomial $p$ in (\ref{eq:opening.poly.ex}) is nonnegative or a sum of squares in fact coincide. In the general situation, however, while sum of squares polynomials are clearly always nonnegative, it is not true that the converse always holds. This was shown in the same paper of Hilbert~\cite{Hilbert_1888}, where he gave a non-constructive proof of existence of nonnegative polynomials that are not sums of squares. Explicit examples of such polynomials appeared many years later, starting with the work of Motzkin~\cite{MotzkinSOS} in the 1960s. Hilbert's interest in this line of research is also showcased in his 17th problem, which asks whether every nonnegative polynomial is a sum of squares of \emph{rational functions}. We refer the interested reader to an outstanding survey paper of Reznick~\cite{Reznick}, which covers many historical aspects around Hilbert's 17th problem, including the affirmative solution by Artin~\cite{Artin_Hilbert17}, as well as several later developments.

The classical questions around nonnegative and sum of squares polynomials have been revisited quite extensively in the past 10-15 years in different communities among applied and computational mathematicians. The reason for this renewed interest is twofold: (i) the discovery that many problems of modern practical interest can be cast as optimization problems over nonnegative polynomials, and (ii) the observation that while optimizing over nonnegative polynomials is generally NP-hard, optimization over the set of sum of squares polynomials can be done via \emph{semidefinite programming} (SDP); see Theorem~\ref{thm:sos.sdp} in Section~\ref{sec:sos.SDP.background}. The latter development, originally explored in the pioneering works of Shor~\cite{Shor}, Nesterov~\cite{NesterovSquared}, Parrilo~\cite{PhD:Parrilo},\cite{sdprelax}, and Lasserre~\cite{lasserre_moment}, has led to the creation of \emph{sum of squares optimization}---a computational framework, with semidefinite programming as its underlying engine, that can tackle many fundamental problems of  real algebra and polynomial optimization.

The dependence of sum of squares approaches on semidefinite programming can be viewed as both a strength and a weakness depending on one's perspective. From a computational complexity viewpoint, semidefinite programs can be solved with arbitrary accuracy in polynomial time using interior point methods (see \cite{VaB:96} for a comprehensive survey). As a result, sum of squares techniques offer polynomial time algorithms that approximate a very broad class of NP-hard problems of interest. From a more practical viewpoint however, SDPs are among the most expensive convex relaxations to solve. The speed and reliability of the current SDP solvers lag behind those for other more restricted classes of convex programs (such as linear or second order cone programs) by a wide margin. With the added complication that the SDPs generated by sos problems are large (see Section~\ref{sec:sos.SDP.background}), scalability has become the single most outstanding challenge for sum of squares optimization in practice.


\aaa{In this paper, we focus on the latter of the two viewpoints mentioned above and offer alternatives to sum of squares optimization that allow one to trade off computation time with solution quality.  While these alternatives are more conservative in general, they are significantly more scalable.}
\aaa{Our hope is that the proposed approaches will} expand the use and applicability of algebraic techniques in optimization to new areas and share its appeal with a broader audience. We call our new computational frameworks, which rely on linear and second order cone programming, \emph{DSOS and SDSOS optimization}\footnote{We use and recommend the pronunciation \emph{``d-sauce''} and \emph{``s-d-sauce''}.}. These are short for \emph{diagonally-dominant-sum-of-squares} and \emph{scaled-diagonally-dominant-sum-of-squares}; see Section~\ref{sec:dsos.sdsos} for precise definitions. While these tools are primarily designed for sum of squares optimization, they are also applicable to general applications of semidefinite programming where tradeoffs between scalability and performance may be desirable. In the interest of motivating our contributions for a diverse audience, we delay a presentation of these contributions until Section~\ref{sec:dsos.sdsos} and start instead with a portfolio of problem areas involving nonnegative polynomials. Any such problem area is one to which sum of squares optimization, as well as its new DSOS and SDSOS counterparts, are directly applicable.


\subsection{Why optimize over nonnegative polynomials?}
\label{sec:intro_applications}

We describe several motivating applications in this section at a high level. These will be revisited later in the paper with concrete computational examples; see Section~\ref{sec:experiments}.

\paragraph{{\bf Polynomial optimization}} A polynomial optimization problem (POP) is an optimization problem of the form

\begin{equation}\label{eq:POP}
\begin{array}{lll}
\mbox{minimize} &  p(x) &\ \\ 
\mbox{subject to} & x\in K\mathrel{\mathop:}=\{x\in\mathbb{R}^n\  |\  g_i(x)\geq 0, h_i(x)=0\}, &\
\end{array}
\end{equation}
where $p$, $g_i$, and $h_i$ are multivariate polynomials. The special case of problem (\ref{eq:POP}) where the polynomials $p,g_i,h_i$ all have degree one is of course \emph{linear programming}, which can be solved very efficiently. For higher degrees, POP contains as special cases many important problems in operations research; e.g., the optimal power flow problem in power engineering~\cite{OPF_survey}, the computation of Nash equilibria in game theory~\cite{lasserre_games},~\cite{Pablo_poly_games}, problems of Euclidean embedding and distance geometry~\cite{leo_siam_dg}, and a host of problems in combinatorial optimization.
%
We observe that if we could optimize over the set of polynomials that are \emph{nonnegative on a closed basic semialgebraic set}, then we could solve the POP problem to global optimality. To see this, note that the optimal value of problem (\ref{eq:POP}) is equal to the optimal value of the following problem:
\begin{equation}\label{eq:POP.equivalent}
\begin{array}{lll}
\mbox{maximize} &  \gamma &\ \\ 
\mbox{subject to} & p(x)-\gamma\geq 0, \ \forall x\in K. &\
\end{array}
\end{equation}
Here, we are trying to find the largest constant $\gamma$ such that the polynomial $p(x)-\gamma$ is nonnegative on the set $K$; i.e., the largest lower bound on problem (\ref{eq:POP}).


\paragraph{{{}\bf Combinatorial optimization}} Proofs of nonnegativity of polynomials of degree as low as four provide \emph{infeasibility certificates} for all decision problems in the complexity class NP, including many well-known combinatorial optimization problems. Consider, e.g., 
the simple-to-describe, NP-complete problem of PARTITION~\cite{GareyJohnson_Book}: Given a set of integers $a_1,\ldots, a_n$, decide if they can be split into two sets with equal sums.
It is straightforward to see that a PARTITION instance is infeasible if and only if the degree-four polynomial 
$$p(x)=\sum_{i=1}^n (x_i^2-1)^2+(\sum_{i=1}^n x_ia_i)^2$$
satisfies $p(x)>0$, $\forall x\in\mathbb{R}^n$. Indeed, $p$ is by construction a sum of squares and hence nonnegative. If it were to have a zero, each square in $p$ would have to evaluate to zero; but this can happen if and only if there is a binary vector $x\in\{-1,1\}^n$, which makes $\sum_{i=1}^n x_ia_i=0$, giving a yes answer to the PARTITION instance. Suppose now that for a given instance, and for some $\epsilon>0$, we could prove that $p(x)-\epsilon$ is \emph{nonnegative}. Then we would have proven that our PARTITION instance is infeasible, potentially without trying out all $2^n$ ways of splitting the integers into two sets. 

\paragraph{{\bf Control systems and robotics}} Numerous fundamental problems in nonlinear dynamics and control, such as stability, robustness, collision avoidance, and controller design can be turned into problems about finding special functions---the so-called \emph{Lyapunov functions} (see, e.g.,~\cite{Khalil:3rd.Ed})---that satisfy certain sign conditions. For example, given a differential equation $\dot{x}=f(x)$, where $f:\mathbb{R}^n\rightarrow\mathbb{R}^n$, and with the origin as an equilibrium point (i.e., satisfying $f(0)=0$), consider the ``region of attraction (ROA) problem'': Determine the set of initial states in $\mathbb{R}^n$ from which the trajectories flow to the origin. 
Lyapunov's stability theorem (see, e.g.,~\cite[Chap. 4]{Khalil:3rd.Ed}) tells us that if we can find a (Lyapunov) function $V:\mathbb{R}^n\rightarrow\mathbb{R}$, which together with its gradient $\nabla V$ satisfies

\begin{equation}\label{eq:poly_Lyap_inequalities}
V(x)>0\quad \forall x\neq0,   \quad \mbox{and} \quad \langle\nabla V(x),f(x)\rangle<0\quad \forall x\in\{x|\ V(x)\leq\beta, x\neq 0\},
\end{equation}
then the set $\{x\in\mathbb{R}^n|\ V(x)\leq\beta\}$ is part of the region of attraction. If $f$ is a polynomial function (a very important case in applications~\cite{PhD:Parrilo}), and if we parameterize $V$ as a polynomial function, then the search for the coefficients of $V$ satisfying the conditions in (\ref{eq:poly_Lyap_inequalities}) is an optimization problem over the set of nonnegative polynomials. Designing stable equilibrium points with large regions of attraction is a fundamental problem in control engineering and robotics; see, e.g.,~\cite{some_control_apps_sos},~\cite{tedrake_roa_walking}.

\paragraph{{\bf Statistical regression with shape constraints}} A problem that arises frequently in statistics is that of fitting a function to a set of data points with minimum error, while ensuring that it meets some structural properties, such as nonnegativity, monotonicity, or convexity~\cite{gupta_monotone_regression},~\cite{hannah_convex_regression}. Requirements of this type are typically imposed either as regularizers (to avoid overfitting), or more importantly as a result of prior knowledge that the true function to be estimated satisfies the same structural properties. In economics, for example, a regression problem for estimating the utility of consumers from sample measurements would come with a natural concavity requirement on the utility function. When the regression functions are polynomials, many such structural properties lead to polynomial nonnegativity constraints. For example, a polynomial $p(x)\mathrel{\mathop:}=p(x_1,\ldots,x_n)$ can be constrained to be concave by requiring its Hessian matrix $\nabla^2p(x)$, which is an $n\times n$ symmetric matrix with polynomial entries, to be negative semidefinite for all $x$. This condition is easily seen to be equivalent to the polynomial $-y^TH(x)y$ in the $2n$ variables $x\mathrel{\mathop:}=(x_1,\ldots,x_n)$ and $y\mathrel{\mathop:}=(y_1,\ldots,y_n)$ being nonnegative.

\paragraph{{\bf Probability bounds given moments}} The well-known inequalities of Markov and Chebyshev in probability theory bound the probability that a univariate random variable takes values in a subset of the real line given the first, or the first and second moments of its distribution~\cite{bertsekas_prob}. Consider a vast generalization of this problem where one is given a finite number of moments $\{\mu_{k_1,\ldots,k_n}\mathrel{\mathop:}=E[X_1^{k_1}\ldots X_n^{k_n}], k_1+\cdots,k_n\leq d\}$ of a multivariate multivariate random variable $X\mathrel{\mathop:}=(X_1,\ldots,X_n)$ and is interested in bounding the probability of an event $X\in\mathcal{S}$, where $\mathcal{S}$ is a general basic semialgebraic subset of the sample space $\Omega$. A sharp upper bound on this probability can be obtained by solving the following optimization problem for the coefficients $c_{k_1\ldots k_n}$ of an $n$-variate degree-$d$ polynomial $p$ (see~\cite{bertsimas_popescu_moment_prob},~\cite{popescu_moment_prob}):

\begin{equation}\label{eq:probability.bound.opt}
\begin{array}{lll}
\mbox{minimize} &  E[p(X)]=\sum c_{k_1\ldots k_n} \mu_{k_1\ldots k_n} &\ \\ 
\mbox{subject to} & p(x)\geq 1, \forall x\in\mathcal{S} &\ \\
\  & p(x)\geq 0, \forall x\in\Omega. &\
\end{array}
\end{equation}

It is easy to see that any feasible solution of problem (\ref{eq:probability.bound.opt}) gives an upper bound on the probability of the event $X\in\mathcal{S}$ as the polynomial $p$ is by construction placed above the indicator function of the set $S$ and the expected value of this indicator function is precisely the probability of the event $X\in\mathcal{S}$. Note that the constraints in (\ref{eq:probability.bound.opt}) are polynomial nonnegativity conditions.

\paragraph{{\bf Other applications}} The application areas outlined above are only a few examples in a long list of problems that can be tackled by optimizing over the set of nonnegative polynomials. Other interesting problems on this list include: computation of equilibria in games with continuous strategy spaces~\cite{Pablo_poly_games}, distinguishing separable and entangled states in quantum information theory~\cite{Pablo_Sep_Entang_States}, computing bounds on sphere packing problems in geometry~\cite{vallentin_bachoc_spheres}, software verification~\cite{MardavijRoozbehani2008}, volume computation~\cite{henrion2009approximate}, robust and stochastic optimization~\cite{DanIancu_sos}, filter design~\cite{filter_design_trig_sos}, combinatorics~\cite{raymond2018symmetry},~\cite{gouveia2010theta}, and automated theorem proving~\cite{harrison2007verifying}. A great reference for many applications in this area, as well as the related theoretical and computational background, is a recent volume edited by Blekherman, Parrilo, and Thomas~\cite{PabloGregRekha_BOOK}.

\paragraph{{\bf Organization of the paper}}  The remainder of this paper is structured as follows. In Section~\ref{sec:sos.SDP.background}, we briefly review the sum of squares approach and its relation to semidefinite programming for the general reader. We also comment on its challenges regarding scalability and prior work that has tackled this issue. The familiar reader can skip to Section~\ref{sec:dsos.sdsos}, where our contributions begin. In Section \ref{subsec:dsos.sdsos.cones}, we introduce the cones of dsos and sdsos polynomials and demonstrate that we can optimize over them using linear and second order cone programming respectively. Section \ref{subsec:rdsos.rsdsos} introduces hierarchies of cones based on dsos and sdsos polynomials that can better approximate the cone of nonnegative polynomials. Asymptotic guarantees on these cones are also presented. Similar to the sum of squares hierarchy, these guarantees can be turned into a hierarchy of converging lower bounds for polynomial optimization problems that use linear or second order cone programming in each step.


 Section \ref{sec:experiments} presents numerical experiments that use our approach on large-scale examples from polynomial optimization (Section \ref{subsec:experiments.pop}), combinatorial optimization (Section \ref{sec:copositivity}), statistics and machine learning (Sections \ref{subsec:experiments.convex.regression} and \ref{subsec:experiments.sparse.pca}), financial mathematics (Section~\ref{subsec:experiments.options.pricing}), and control theory and robotics (Section \ref{sec:controls_apps}). Section \ref{sec:improvements} provides a review of recent techniques for bridging the gap between the DSOS/SDSOS approach and the SOS approach at the expense of additional computational costs. Section \ref{sec:conclusions} concludes the paper. The Supplementary Material introduces \texttt{iSOS} (``inside SOS''), an accompanying MATLAB package for DSOS and SDSOS optimization written using the SPOT toolbox \cite{Megretski_spot}. 
 


\section{Review of the semidefinite programming-based approach and computational considerations} \label{sec:sos.SDP.background}
As mentioned earlier, sum of squares optimization handles a (global) nonnegativity constraint on a polynomial $p$ by replacing it with the \ani{more restrictive} requirement that $p$ be a sum of squares. The situation where $p$ is only constrained to be nonnegative on a certain basic semialgebraic set\footnote{In this formulation, we have avoided equality constraints for simplicity. Obviously, there is no loss of generality in doing this as an equality constraint $h(x)=0$ can be imposed by the pair of inequality constraints $h(x)\geq 0, -h(x)\geq 0$.} $$\mathcal{S}\mathrel{\mathop:}=\{x\in\mathbb{R}^n| \ g_1(x)\geq 0,\ldots,g_m(x)\geq 0\}$$ can also be handled with the help of appropriate sum of squares multipliers. For example, if we succeed in finding sos polynomials $s_0,s_1,\ldots,s_m$, such that

\begin{equation} \label{eq:p=s0+sigi}
p(x)=s_0(x)+\sum_{i=1}^m s_i(x)g_i(x),
\end{equation}
then we have found a certificate of nonnegativity of $p$ on the set $\mathcal{S}$. Indeed, if we evaluate the above expression at any $x\in\mathcal{S}$, nonnegativity of the polynomials $s_0,s_1\ldots,s_m$ imply that $p(x)\geq 0$. A
Positivstellensatz from real algebraic geometry due to Putinar~\cite{putinar1993positive} states that if the set $\mathcal{S}$ satisfies the so-called \emph{Archimedean} property, a property only slightly stronger than compactness (see, e.g.,~\cite{Laurent_survey} for a precise definition), then every polynomial positive on $\mathcal{S}$ has a representation of the type (\ref{eq:p=s0+sigi}), for some sos polynomials $s_0,s_1,\ldots,s_m$ of high enough degree (see also~\cite{nie2007Putinarcomplexity} for degree bounds). Even with no qualifications whatsoever regarding the set $\mathcal{S}$, there are other Positivstellens\"atze (e.g., due to Stengle~\cite{stengle1974nullstellensatz}) that certify nonnegativity of a polynomial on a basic semialgebraic set using sos polynomials. These certificates are only slightly more complicated than (\ref{eq:p=s0+sigi}) and involve sos multipliers associated with products among polynomials $g_i$ that define $\mathcal{S}$~\cite{sdprelax}. A great reference for the interested reader is the survey paper by Laurent~\cite{Laurent_survey}.\footnote{Our review material in this section does not do justice to the ``dual'' theory of \emph{generalized moment problems}; in addition to the survey paper by Laurent, the interested reader is referred to the recent textbooks by Lasserre~\cite{Lasserre_book},~\cite{Lasserre_Book2} and references therein.}

The computational advantage of a certificate of (global or local) nonnegativity via sum of squares polynomials is that it can be automatically found by semidefinite programming. The following well-known theorem establishes the link between sos polynomials and SDP.
%
%
Let us denote the set of $n\times n$ symmetric matrices by $\mathcal{S}_n$. Recall that a matrix $A\in\mathcal{S}_n$ is \emph{positive semidefinite} (psd) if $x^TAx\geq 0, \forall x\in\mathbb{R}^n$, and that semidefinite programming is the problem of optimizing a linear function over psd matrices subject to affine inequalities on their entries~\cite{VaB:96}. We denote the positive semidefiniteness of a matrix $A$ with the standard notation $A\succeq 0$.


\begin{theorem}[see, e.g., \cite{PhD:Parrilo},\cite{sdprelax}]
	\label{thm:sos.sdp}
	A multivariate polynomial $p\mathrel{\mathop:}=p(x)$ in $n$ variables and of degree
	$2d$ is a sum~of~squares if and only if there exists a symmetric matrix $Q$ (often called the Gram matrix) such that
	\begin{equation}\label{eq:p=z'Qz}
	\begin{array}{rll}
	p(x)&=&z^{T}Qz, \\
	Q&\succeq&0,
	\end{array}
	\end{equation}
	where $z$ is the vector of monomials of degree up to $d$:
	\begin{equation*}\label{eq:monomials}
	z=[1,x_{1},x_{2},\ldots,x_{n},x_{1}x_{2},\ldots,x_{n}^d].
	\end{equation*}
\end{theorem}

Searching for a matrix $Q$ satisfying the positive semidefiniteness constraint, as well as the linear equality constraints coming from (\ref{eq:p=z'Qz}), amounts to solving a semidefinite program. The size of the matrix $Q$ in this theorem is $${n+d\choose d}\times {n+d\choose d},$$
which approximately equals $n^d \times n^d$. While this number is polynomial in $n$ for fixed $d$, it can grow rather quickly even for low-degree polynomials. For example, a degree-$4$ polynomial ($d=2$) in $50$ variables has 316251 coefficients and its Gram matrix, which would need to be positive semidefinite, contains 879801 decision variables. A semidefinite constraint of this size is quite expensive, and in fact a problem with $50$ variables is well beyond the current realm of possibilities in SOS optimization. In the absence of problem structure, sum of squares problems involving degree-$4$ or $6$ polynomials are currently limited, roughly speaking, to a handful or a dozen variables.

There have been many contributions already to improvements in scalability of sum of squares techniques. One approach has been to develop systematic techniques for taking advantage of problem structure, such as sparsity or symmetry of the underlying polynomials, to reduce the size of the SDPs~\cite{Symmetry_groups_Gatermann_Pablo, vallentin_symmetry,deklerk_symmetry,riener_symmetry}. These techniques have proven to be very useful as problems arising in practice sometimes do come with a \aaa{lot of structure}. 
Another approach which holds promise has been to design customized solvers for SOS programs that avoid resorting to an off-the-shelf interior point solver. These techniques can often lead to improvements for special classes of problems. Examples in this direction include the works in~\cite{bertsimas_large_scale, nie_large_scale,sun_large_scale,qsdpnal,zheng2017fast,yang2015sdpnal,papp2017sum,henrion2012projection}. There has also been recent work by Lasserre et al. that increases scalability of sum of squares optimization problems at the cost of accuracy of the solutions obtained. This is done by bounding the size of the largest SDP constraint appearing in the sos formulation, and leads to what the authors call the BSOS (bounded SOS) hierarchy~\cite{Lasserre_bsos}.




The approach we take in this paper for enhancing scalability is orthogonal to the ones mentioned above (and can potentially later be combined with them). 
We propose to eschew sum of squares decompositions to begin with. In our view, almost all application areas of this field use sum of squares decompositions not for their own sake, but because they provide a means to polynomial nonnegativity. Hence, this paper is motivated by a natural question: Can we give other sufficient conditions for polynomial nonnegativity that are perhaps \ani{more restrictive} than a sum of squares \aaa{decomposition}, but cheaper to work with?\footnote{Given this motivation, we note that our numerical experiments in Section 4 compare our approach primarily with the standard and widely-used semidefinite programming-based approach to sum of squares optimization. They are certainly not meant to be exhaustive. We believe that comparisons with the special-purpose techniques highlighted above and with other solvers would be interesting but are better suited for a separate work given our space limitations.
}

In the next section, we identify some subsets of the cone of nonnegative polynomials that one can optimize over using linear programming (LP) and second order cone programming (SOCP)~\cite{socp_boyd},~\cite{socp_alizadeh_goldfarb}. Not only are these much more efficiently solvable classes of convex programs, but they are also superior to semidefinite programs in terms of numerical conditioning. In addition, working with these classes of convex programs allows us to take advantage of high-performance LP and SOCP solvers (e.g., CPLEX~\cite{cplex}) that have been matured over several decades because of industry applications. We remark that while there have been other approaches to produce LP hierarchies for polynomial optimization problems (e.g., based on the Krivine-Stengle certificates of positivity \cite{Krivine, stengle1974nullstellensatz, Lasserre_book}), these LPs, though theoretically significant, are typically quite weak in practice and often numerically ill-conditioned \cite{Lasserre_bsos}. Finally, we remark that for a range of combinatorial optimization problems, there exist interesting algebraic approaches that do not rely on SDPs and only involve linear algebraic operations~\cite{de2011computing},~\cite{de2008hilbert}. While our techniques are similar in the spirit of avoiding SDPs, they are not tied to combinatorial (or polynomial) optimization problems.

\section{DSOS and SDSOS Optimization}\label{sec:dsos.sdsos}
We start with some basic definitions. A \emph{monomial} $m:\mathbb{R}^n\rightarrow\mathbb{R}$ in variables $x\mathrel{\mathop:}=(x_1,\ldots,x_n)$ is a function of the form $m(x)=x_1^{\alpha_1}\ldots x_n^{\alpha_n}$, where each $\alpha_i$ is a nonnegative integer. The \emph{degree} of such a monomial is by definition $\alpha_1+\cdots+\alpha_n$. A \emph{polynomial} $p:\mathbb{R}^n\rightarrow\mathbb{R}$ is a finite linear combination of monomials $p(x)=\sum_{{\alpha_1}\ldots{\alpha_n}} c_{{\alpha_1}\ldots{\alpha_n}} x_1^{\alpha_1}\ldots x_n^{\alpha_n}$. The degree of a polynomial is the maximum degree of its monomials.  A \emph{form} or a \emph{homogeneous polynomial} is a polynomial whose monomials all have the same degree. For any polynomial $p\mathrel{\mathop:}=p(x_1,\ldots,x_n)$ of degree $d$, one can define its \emph{homogenized version} $p_h$ by introducing a new variable $y$ and defining $p_h(x_1,\ldots,x_n,y)=y^d p(\frac{x}{y})$. One can reverse this operation (dehomogenize) by setting the homogenizing variable $y$ equal to one: $p(x_1,\ldots,x_n)=p_h(x_1,\ldots,x_n,1)$. The properties of being nonnegative and a sum of squares (as defined earlier) are preserved under homogenization and dehomogenization~\cite{Reznick}. 
We say that a form $f$ is \emph{positive definite} if $f(x)>0$ for all $x\neq 0$. We denote by $PSD_{n,2d}$ and $SOS_{n,2d}$ the set of nonnegative (a.k.a. positive semidefinite) and sum of squares polynomials in $n$ variables and degree $2d$ respectively. (Note that odd-degree polynomials can never be nonnegative.) Both these sets form proper (i.e., convex, closed, pointed, and solid) cones in $\mathbb{R}^{n+2d\choose 2d}$ with the obvious inclusion relationship $SOS_{n,2d}\subseteq\ PSD_{n,2d}$.

The basic idea behind our approach is to approximate $SOS_{n,2d}$ from the inside with new sets that are more tractable for optimization purposes. Towards this goal, one may think of several natural sufficient conditions for a polynomial to be a sum of squares. For example, the following may be natural first candidates:
\begin{itemize}
	\item The set of polynomials that are sums of 4-th powers of polynomials: \\$\{p| \ p=\sum q_i^4\}$,
	\item The set of polynomials that are a sum of a constant number of squares of polynomials; e.g., a sum of three squares: $\{p|\  p=q_1^2+q_2^2+q_3^2\}.$
\end{itemize} 

Even though both of these sets clearly reside inside the set of sos polynomials, they are not any easier to optimize over. In fact, they are much harder: testing whether a polynomial is a sum of 4-th powers is NP-hard already for quartics~\cite{tensor_hard} (in fact, the cone of 4-th powers of linear forms is dual to the cone of nonnegative quartic forms~\cite{reznick_blenders}) and optimizing over polynomials that are sums of three squares is intractable (as this task even for quadratics subsumes the NP-hard problem of positive semidefinite matrix completion with a rank constraint~\cite{laurent_completion_rank}). These examples illustrate that in general, an inclusion relationship does not imply anything with respect to optimization complexity. Hence, we need to take some care when choosing which subsets of $SOS_{n,2d}$ to work with. On the one hand, these subsets have to be ``big enough'' to be useful in practice; on the other hand, they should be more tractable to optimize over.



\subsection{The cones of dsos and sdsos polynomials} \label{subsec:dsos.sdsos.cones}


We now describe two interesting cones inside $SOS_{n,2d}$ that lend themselves to linear and second order cone representations and are hence more tractable for optimization purposes. These cones will also form the building blocks of some more elaborate cones (with improved performance) that we will present in Subsection~\ref{subsec:rdsos.rsdsos} and Section~\ref{sec:improvements}.


\begin{definition} \label{def:dsos} 
	A polynomial $p\mathrel{\mathop:}=p(x)$ is \emph{diagonally-dominant-sum-of-squares} (dsos) if it can be written as 
	\begin{equation}\label{eq:dsos}
	p(x)=\sum_i \alpha_i m_i^2(x) + \sum_{i,j} \beta_{ij}^+ (m_i(x)+ m_j(x))^2+ \sum_{i,j} \beta_{ij}^- (m_i(x)- m_j(x))^2,
	\end{equation}
	for some monomials $m_i(x), m_j(x)$ and some nonnegative scalars $\alpha_i,\beta_{ij}^+, \beta_{ij}^-$. We denote the set of polynomials in $n$ variables and degree $2d$ that are dsos by $DSOS_{n,2d}$.
\end{definition}

\begin{definition} \label{def:sdsos} 
	A polynomial $p\mathrel{\mathop:}=p(x)$ is \emph{scaled-diagonally-dominant-sum-of-squares} (sdsos) if it can be written as 
	\begin{equation}\label{eq:sdsos}
	p(x)=\sum_i \alpha_i m_i^2(x) + \sum_{i,j}  (\hat{\beta}_{ij}^+ m_i(x)+ \tilde{\beta}_{ij}^+ m_j(x))^2+\sum_{i,j}  (\hat{\beta}_{ij}^- m_i(x)- \tilde{\beta}_{ij}^- m_j(x))^2,
	\end{equation}
	for some monomials $m_i(x), m_j(x)$ and some scalars $\alpha_i, \hat{\beta}_{ij}^+, \tilde{\beta}_{ij}^+, \hat{\beta}_{ij}^-, \tilde{\beta}_{ij}^-$ with $\alpha_i \geq~0$. We denote the set of polynomials in $n$ variables and degree $2d$ that are sdsos by $SDSOS_{n,2d}$. 
\end{definition}


From the definition, the following inclusion relationships should be clear:
$$DSOS_{n,2d}\subseteq SDSOS_{n,2d}\subseteq SOS_{n,2d}\subseteq PSD_{n,2d}.$$
In general, all these containment relationships are strict. Figure~\ref{fig:dsos.sdsos.sos} shows these sets for a family of bivariate quartics parameterized by two coefficients $a,b$:\footnote{It is well known that all psd bivariate forms are sos~\cite{Reznick} and hence the outer set exactly characterizes all values of $a$ and $b$ for which this polynomial is nonnegative.}  

\begin{equation}\label{eq:parametric.bivariate.quartic}
p(x_1,x_2)=x_1^4+x_2^4+ax_1^3x_2+(1-\frac12 a-\frac12 b)x_1^2x_2^2+2bx_1x_2^3.
\end{equation}

 Some remarks about the two definitions above are in order. It is not hard to see that if $p$ has degree $2d$, the monomials $m_i,m_j$ in (\ref{eq:dsos}) or (\ref{eq:sdsos}) never need to have degree higher than $d$. In the special case where $p$ is homogeneous of degree $2d$ (i.e., a form), the monomials $m_i, m_j$ can be taken to have degree exactly $d$. We also note that decompositions given in (\ref{eq:dsos}) or (\ref{eq:sdsos}) are not unique; there can even be infinitely many such decompositions but the definition requires just one. \aaa{Finally, we have found out that interestingly, sdsos polynomials were studied in the 1970's and 1980's for purely algebraic reasons by Robinson~\cite{RobinsonSOS}, Reznick~\cite{reznick_AGI}, and Choi, Lam, and Reznick~\cite{even_sym_sextics}. These papers use the more descriptive terminology of \emph{``sum of binomial squares''} (sbs), which is consistent with the decomposition in (\ref{eq:sdsos}). In~\cite{reznick_AGI} for example, Reznick characterizes interesting families of forms that are sos if and only if they are sbs (i.e., sdsos). In~\cite{even_sym_sextics}, Choi, Lam, and Reznick provide a complete description of sbs forms and their extremal elements among the family of even symmetric sextics. In the process, they answer questions that Robinson had raised earlier~\cite{RobinsonSOS} about sbs polynomials.}
 
 

\begin{figure}
	\centering
		\includegraphics[width=.5\columnwidth]{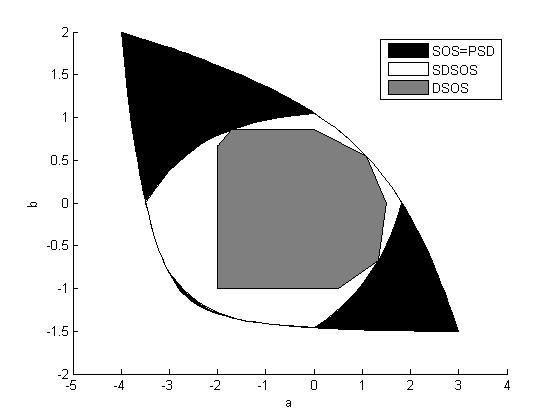}
	\caption{A comparison of the set of dsos/sdsos/sos/psd polynomials on a parametric family of bivariate quartics given in (\ref{eq:parametric.bivariate.quartic}).}
	\label{fig:dsos.sdsos.sos}
\end{figure}


%
%
%
%
%
%
%
%
%
%
%
%

Our terminology in Definitions~\ref{def:dsos} and~\ref{def:sdsos} comes from a connection (which we will soon establish) to the following classes of symmetric matrices.
\begin{definition}
	A symmetric matrix $A=(a_{ij})$ is \emph{diagonally dominant} (dd) if $$a_{ii} \geq \sum_{j \neq i} |a_{ij}|$$ for all $i$. A symmetric matrix $A$ is \emph{scaled diagonally dominant} (sdd) if there exists a diagonal matrix $D$, with positive diagonal entries, such that $DAD$ is dd.\footnote{\aaa{The requirement that the diagonal matrix $D$ have positive diagonal entries can be replaced with it being nonsingular without changing the definition.} Checking whether a given matrix $A=(a_{ij})$ is sdd can be done by solving a linear program since the definition is equivalent to existence of an element-wise positive vector $d$ such that
		$$a_{ii} d_i \geq \sum_{j \neq i} |a_{ij}| d_j, \forall i.$$} 
	 We denote the set of $n \times n$ dd and sdd matrices with $DD_n$ and $SDD_n$ respectively.
\end{definition}

It follows from Gershgorin's circle theorem~\cite{gersh} that diagonally dominant matrices are positive semidefinite. This implies that scaled diagonally dominant matrices are also positive semidefinite as the eigenvalues of $DAD$ have the same sign as those of $A$. Hence, if we denote the set of $n\times n$ positive semidefinite matrices by $P_n$, the following inclusion relationships are evident:

$$DD_n\subseteq SDD_n \subseteq P_n.$$
These containments are strict for $n>2$. In Figure~\ref{fig:dd.sdd.psd}, we illustrate the set of $x$ and $y$ for which the matrix $$I+xA+yB$$ is diagonally dominant, scaled diagonally dominant, and positive semidefinite. Here, $I$ is the identity matrix and the $5 \times 5$ symmetric matrices $A$ and $B$ were generated randomly with iid entries sampled from the standard normal distribution. Our interest in these inner approximations to the set of psd matrices stems from the fact that optimization over them can be done by linear and second order cone programming respectively (see Theorem~\ref{thm:dsos.sdsos.LP.SOCP} below, which is more general). For now, let us relate these matrices back to dsos and sdsos polynomials.

\begin{figure}
	\centering
	\includegraphics[width=.5\columnwidth]{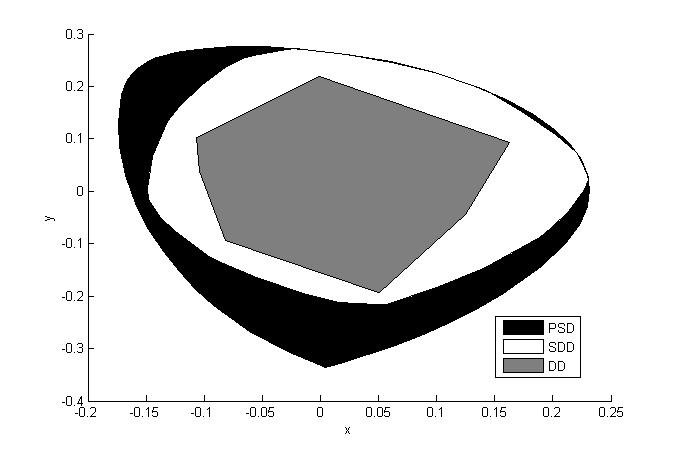}
	\caption{A section of the cone of $5\times 5$ diagonally dominant, scaled diagonally dominant, and positive semidefinite matrices. Optimization over these sets can respectively be done by LP, SOCP, and SDP.}
	\label{fig:dd.sdd.psd}
\end{figure}

\begin{theorem}\label{thm:dsos.dd}
	A polynomial $p$ of degree $2d$ is dsos if and only if it admits a representation as $p(x)=z^T(x)Qz(x)$, where $z(x)$ is the standard monomial vector of degree $\leq d$ and $Q$ is a dd matrix.
\end{theorem}

The following lemma by Barker and Carlson provides an extreme ray characterization of the set of diagonally dominant matrices and will be used in the proof of the above theorem.

\begin{lemma}[Barker and Carlson \cite{dd_extreme_rays}]\label{lem:dd.corners}
Let \aaa{$\mathcal{V}=\{v_j\}_{j=1}^{2n^2}$} be the set of all nonzero vectors in $\mathbb{R}^n$ with at most $2$ nonzero components, each equal to $\pm 1$. \aaa{Let $\{V_i\}_{i=1}^{n^2}$ be the set of all rank-one matrices of the form $V_i=vv^T$ for some $v\in\mathcal{V}.$} Then, a symmetric $n\times n$ matrix $M$ is diagonally dominant if and only if it can be written as $$M=\sum_{i=1}^{n^2} \eta_i \aaa{V_i}, $$ for some $\eta_i\geq 0$.
\end{lemma}

\begin{proof}[Proof of Theorem~\ref{thm:dsos.dd}] Suppose first that $p(x)=z^T(x)Qz(x)$ with $Q$ diagonally dominant. Lemma~\ref{lem:dd.corners} implies that $$p(x)=\sum_{i=1}^{n^2} \left( \eta_i z^T(x) v_i v_i^T z(x)\right)=\sum_{i=1}^{n^2} \eta_i (v_i^T z(x))^2 ,$$ where $\eta_i \geq 0$ and $\{v_i\}$ is the set of all {nonzero} vectors in $\mathbb{R}^n$ with at most $2$ nonzero components, each equal to $\pm 1$. In this sum, the vectors $v_i$ that have only one nonzero entry lead to the first sum in the dsos expansion (\ref{eq:dsos}). The vectors $v_i$ with two $1$s or two $-1$s lead to the second sum, and those with one $1$ and one $-1$ lead to the third.

For the reverse direction, suppose $p$ is dsos, i.e., has the representation in (\ref{eq:dsos}). We will show that $$p(x)=z^T(x)\sum_k Q_k z(x),$$ where each $Q_k$ is dd and corresponds to a single term in the expansion (\ref{eq:dsos}). As the sum of dd matrices is dd, this would establish the proof. For each term of the form $\alpha_i m_i^2(x)$ we can take $Q_k$ to be a matrix of all zeros except for a single diagonal entry corresponding to the monomial $m_i$ in $z(x)$, with this entry equaling $\alpha_i$. For each term of the form $\beta_{ij}^+ (m_i(x)+ m_j(x))^2$, we can take $Q_k$ to be a matrix of all zeros except for four \aaa{entries} corresponding to the monomials $m_i(x)$ and $m_j(x)$ in $z(x)$. All these four entries will be set to $\beta_{ij}^+$. Similarly, for each term of the form $\beta_{ij}^- (m_i(x)- m_j(x))^2$, we can take $Q_k$ to be a matrix of all zeros except for four \aaa{entries} corresponding to the monomials $m_i$ and $m_j(x)$ in $z(x)$. In this $2 \times 2$ submatrix, the diagonal elements will be $\beta_{ij}^-$ and the off-diagonal elements will be $-\beta_{ij}^-$. Clearly, all the $Q_k$ matrices we have constructed are dd.
\end{proof}

\begin{theorem}\label{thm:sdsos.sdd}
A polynomial $p$ of degree $2d$ is sdsos if and only if it admits a representation as $p(x)=z^T(x)Qz(x)$, where $z(x)$ is the standard monomial vector of degree $\leq d$ and $Q$ is a sdd matrix. 
\end{theorem}

\begin{proof}
 Suppose first that $p(x)=z^T(x)Qz(x)$ with $Q$ scaled diagonally dominant. Then, there exists a diagonal matrix $D$, with positive diagonal entries, such that $DQD$ is diagonally dominant and $$p(x)=(D^{-1}z(x))^TDQD(D^{-1}z(x)).$$ Now an argument identical to that in the first part of the proof of Theorem~\ref{thm:dsos.dd} (after replacing $z(x)$ with $D^{-1}z(x)$) gives the desired sdsos representation in (\ref{eq:sdsos}).
 
 For the reverse direction, suppose $p$ is sdsos, i.e., has the representation in (\ref{eq:sdsos}). We show that $$p(x)=z^T(x)\sum_k Q_k z(x),$$ where each $Q_k$ is sdd and corresponds to a single term in the expansion (\ref{eq:sdsos}). As the sum of sdd matrices is sdd,\footnote{This claim is not obvious from the definition but is apparent from Lemma~\ref{lem:sdd=sum.2x2} below which implies that $SDD_n$ is a convex \aaa{(in fact, proper)} cone.} this would establish the proof. Indeed, for each term of the form $\alpha_i m_i^2(x)$ we can take $Q_k$ (once again) to be a matrix of all zeros except for a single diagonal entry corresponding to the monomial $m_i$ in $z(x)$, with this entry equaling $\alpha_i$.

 For each term of the form $(\hat{\beta}_{ij}^+ m_i(x)+ \tilde{\beta}_{ij}^+ m_j(x))^2$, we can take $Q_k$ to be a matrix of all zeros except for four \aaa{entries} corresponding to the monomials $m_i(x)$ and $m_j(x)$ in $z(x)$. This $2\times 2$ block will then be $\begin{pmatrix}
 \hat{\beta}_{ij}^+  \\\tilde{\beta}_{ij}^+ \end{pmatrix} \begin{pmatrix}
  \hat{\beta}_{ij}^+  &\tilde{\beta}_{ij}^+ \end{pmatrix} $.  Similarly, for each term of the form $ (\hat{\beta}_{ij}^- m_i(x)- \tilde{\beta}_{ij}^- m_j(x))^2$, we can take $Q_k$ to be a matrix of all zeros except for four \aaa{entries} corresponding to the monomials $m_i$ and $m_j(x)$ in $z(x)$. This $2\times 2$ block will then be $\begin{pmatrix}
   \hat{\beta}_{ij}^-  \\-\tilde{\beta}_{ij}^- \end{pmatrix} \begin{pmatrix}
    \hat{\beta}_{ij}^-  &-\tilde{\beta}_{ij}^- \end{pmatrix} $. It remains to show that a $2\times 2$ rank-1 positive semidefinite matrix is sdd.\footnote{This also implies that any $2\times 2$ positive semidefinite matrix is sdd.} Let $A=\begin{pmatrix}
     a \\b \end{pmatrix} \begin{pmatrix}
      a  &b \end{pmatrix} $ be such a matrix. Observe that $$\begin{pmatrix}\frac{1}{a} &0\\0& \frac{1}{b} \end{pmatrix} A \begin{pmatrix} \frac{1}{a} &0\\0& \frac{1}{b} \end{pmatrix} =\begin{pmatrix} 1& 1\\1& 1 \end{pmatrix}   $$ is dd and hence $A$ is by definition sdd. (Note that if $a$ or $b$ are zero, then $A$ is already dd.)
\end{proof}

The following characterizations of sdd matrices will be important for us.

\begin{theorem}[see theorems 8 and 9 by Boman et al. \cite{sdd_factorwidth}]\label{thm:sdd.factorwidth2}
A symmetric matrix $Q$ is sdd if and only if it has ``factor width'' at most 2; i.e., a factorization $Q=VV^T$, where each column of $V$ has at most two nonzero entries.
\end{theorem}

\begin{lemma}\label{lem:sdd=sum.2x2}
A symmetric $n\times n$ matrix $Q$ is sdd if and only if it can be expressed as 	\begin{equation}\label{eq:sdd.Q=sum.2x2}
Q = \sum_{i< j} M^{ij},
\end{equation}
where each $ M^{ij}$ is an $n\times n$ matrix with zeros everywhere except for four entries $(M^{ij})_{ii}$, $(M^{ij})_{ij}$, $(M^{ij})_{ji}$, $(M^{ij})_{jj}$, which make the $2\times 2$ matrix $\begin{bmatrix} (M^{ij})_{ii} & (M^{ij})_{ij} \\ (M^{ij})_{ji} & (M^{ij})_{jj} \end{bmatrix}$ symmetric and positive semidefinite.
\end{lemma}
\begin{proof}
This is almost immediate from Theorem~\ref{thm:sdd.factorwidth2}: If $Q=VV^T$ for some $n\times k$ matrix $V$, then $Q=\sum_{l=1}^k v_l v_l^T,$ where $v_l$ denotes the $l$-th column of $V$. Since each column $v_l$ has at most two nonzero entries, say in positions $i$ and $j$, each matrix $v_lv_l^T$ will be exactly of the desired form $M^{ij}$ in the statement of the lemma. Conversely, if $Q$ can be written as in (\ref{eq:sdd.Q=sum.2x2}), then we can write each $M^{ij}$ as $M^{ij}=w_{ij,1}w_{ij,1}^T+w_{ij,2}w_{ij,2}^T,$ where the vectors $w_{ij,1}$ and $w_{ij,2}$ have at most two nonzero entries. If we then construct a matrix $V$ which has the collection of the vectors $w_{ij,1}$ and $w_{ij,2}$ as columns, we will have $Q=VV^T.$
\end{proof}

\begin{theorem}\label{thm:dsos.sdsos.LP.SOCP}
For any fixed $d$, optimization over $DSOS_{n,2d}$ (resp. $SDSOS_{n,2d}$) can be done with a linear program (resp. second order cone program) of size polynomial in $n$.
\end{theorem}

\begin{proof}
We will use the characterizations of dsos/sdsos polynomials given in theorems~\ref{thm:dsos.dd} and~\ref{thm:sdsos.sdd}. In both cases, the equality $p(x)=z^T(x)Qz(x),\forall x$ can be imposed by a finite set of \emph{linear equations} in  the coefficients of $p$ and the entries of $Q$ (these match the coefficients of $p(x)$ with those of $z^T(x)Qz(x)$). The constraint that $Q$ be dd can be imposed, e.g., by a set of linear inequalities 
\begin{equation}
\nonumber
\begin{array}{ll}
\ & Q_{ii}\geq \sum_{j\neq i} z_{ij},\forall i, \\
\ & -z_{ij}\leq Q_{ij}\leq z_{ij}, \forall i,j, i\neq j
\end{array}
\end{equation}
in variables $Q_{ij}$ and $z_{ij}$. This gives a linear program. 

The constraint that $Q$ be sdd can be imposed via Lemma~\ref{lem:sdd=sum.2x2} by a set of equality constraints that enforce equation (\ref{eq:sdd.Q=sum.2x2}). The constraint that each $2\times 2$ matrix $\begin{bmatrix} (M^{ij})_{ii} & (M^{ij})_{ij} \\ (M^{ij})_{ji} & (M^{ij})_{jj} \end{bmatrix}$ be psd is a \emph{``rotated quadratic cone''} constraint and can be imposed using SOCP~\cite{socp_alizadeh_goldfarb, socp_boyd}:
$$(M^{ij})_{ii}+(M^{ij})_{jj}\geq 0, ~\Bigl\lvert\Bigl\lvert\begin{pmatrix}
2(M^{ij})_{ij}\\(M^{ij})_{ii}-(M^{ij})_{jj}
\end{pmatrix}\Bigl\lvert\Bigl\lvert \leq (M^{ij})_{ii}+(M^{ij})_{jj}.$$
In both cases, the final LP and SOCP are of polynomial size because the size of the Gram matrix is ${n+d\choose d}$, which is polynomial in $n$ for fixed $d$.
\end{proof}


We have written publicly-available code that automates the process of generating LPs (resp. SOCPs) from dsos (resp. sdsos) constraints on a polynomial. More information about this can be found in the Supplementary Material.

 The ability to replace SDPs with LPs and SOCPs is what results in significant speedups in our numerical experiments (see Section~\ref{sec:experiments}). For the special case where the polynomials involved are quadratic forms, we get a systematic way of inner approximating semidefinite programs. A generic SDP that minimizes $ \mbox{Tr}(CX)$ subject to the constraints $\mbox{Tr}(A_iX)=b_i, i=1,\ldots,m,$ and $X\succeq 0,$ can be replaced by\footnote{We are grateful to Leo Liberti for suggesting the terminology ``DDP and SDDP''.}

\begin{itemize}
\item a \emph{diagonally dominant program} (DDP); i.e., a problem of the form
\begin{flalign} 
\label{eq:ddp}
			\underset{X\in\mathcal{S}_n}{\text{minimize}} \hspace*{1cm} & \mbox{Tr}(CX)  \\
			\text{s.t.} \hspace*{1cm} & \mbox{Tr}(A_iX)=b_i,\quad  i=1,\ldots,m, \nonumber \\
			& X \quad \mbox{dd}, \nonumber
\end{flalign}

which is an LP, or

\item a \emph{scaled diagonally dominant program} (SDDP); i.e., a problem of the form
\begin{flalign} 
\label{eq:sddp}
			\underset{X\in\mathcal{S}_n}{\text{minimize}} \hspace*{1cm} & \mbox{Tr}(CX)  \\
			\text{s.t.} \hspace*{1cm} & \mbox{Tr}(A_iX)=b_i,\quad  i=1,\ldots,m, \nonumber \\
			& X \quad \mbox{sdd}, \nonumber
\end{flalign}
which is an SOCP.

\end{itemize}


\aaa{DDP and SDDP fit nicely within the framework of \emph{conic programming}~\cite[Chap. 2]{bental_nemirov_book} as they are optimization problems over the intersection of an affine subspace with a proper cone ($DD_n$ or $SDD_n$). The description of the associated dual cones can be found e.g. in~\cite{permenter_parrilo_facial_reduction},~\cite{basis_pursuit_Ahmadi_Hall}. In Section~\ref{subsec:experiments.options.pricing} and Section~\ref{subsec:experiments.sparse.pca}, we show some applications of DDP/SDDP to problems in finance and statistics.} 


\subsection{The cone of r-dsos and r-sdsos polynomials and asymptotic guarantees} \label{subsec:rdsos.rsdsos} We now present a hierarchy of cones based on the notions of dsos and sdsos polynomials that can be used to better approximate the cone of nonnegative polynomials.

\begin{definition} \label{def:rdsos.rsdsos} 
For an integer $r\geq 0$, we say that a polynomial $p\mathrel{\mathop:}=p(x_1,\ldots,x_n)$ is 
\emph{r-dsos} (resp. \emph{r-sdsos}) if $$p(x)\cdot(\sum_{i=1}^n x_i^2)^r$$ is dsos (resp. sdsos). We denote the set of polynomials in $n$ variables and degree $2d$ that are r-dsos (resp. r-sdsos) by $rDSOS_{n,2d}$ (resp. $rDSOS_{n,2d}$).
\end{definition}

Note that for $r=0$ we recover our dsos/sdsos definitions. Moreover, because the multiplier $(\sum_{i=1}^n x_i^2)^r$ is nonnegative, we see that for any $r$, the property of being r-dsos or r-sdsos is a sufficient condition for nonnegativity:
$$rDSOS_{n,2d}\subseteq rSDSOS_{n,2d} \subseteq PSD_{n,2d}.$$

Optimizing over r-dsos (resp. r-sdsos) polynomials is still an LP (resp. SOCP). The proof of the following theorem is identical to that of Theorem~\ref{thm:dsos.sdsos.LP.SOCP} and hence omitted.

\begin{theorem}
For any fixed $d$ and $r$, optimization over the set $rDSOS_{n,d}$ (respectively $rSDSOS_{n,d}$) can be done with linear programming (resp. second order cone programming) of size polynomial in $n$.
\end{theorem}

If we revisit the parametric family of bivariate quartics in (\ref{eq:parametric.bivariate.quartic}), the improvements obtained by going one level higher in the hierarchy are illustrated in Figure~\ref{fig:1dsos.1sdsos}. Interestingly, for $r\geq 1$, r-dsos and r-sdsos tests for nonnegativity can sometimes outperform the sos test. The following two examples illustrate this.

\begin{figure}
\begin{center}
    \mbox{
      \subfigure[The LP-based r-dsos hierarchy.]
      {\label{subfig:1dsos}\scalebox{0.25}{\includegraphics{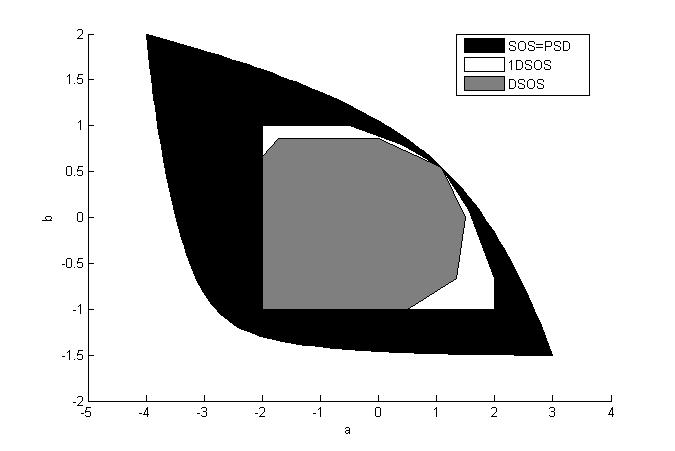}}}}
\mbox{
      \subfigure[The SOCP-based r-sdsos hierarchy.]
      {\label{subfig:1sdsos}\scalebox{0.25}{\includegraphics{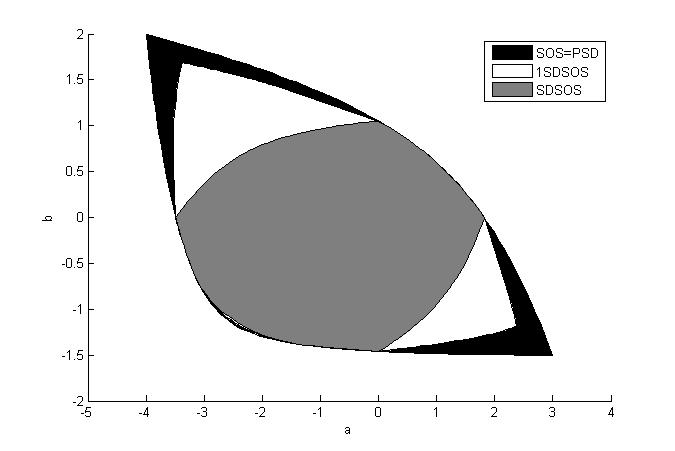}}}
} 

    \caption{The improvement obtained by going one level up in the r-dsos/r-sdsos hierarchies for the family of bivariate quartics in (\ref{eq:parametric.bivariate.quartic}).}
\label{fig:1dsos.1sdsos}
\end{center}
\vspace{-20pt}
\end{figure}


\begin{Example}
	\ani{Consider the Motzkin polynomial, $M(x)=x_1^4x_2^2+x_1^2x_2^4-3x_1^2x_2^2x_3^2+x_3^6,$ which historically is the first known example of a nonnegative polynomial that is \emph{not} a sum of squares~\cite{MotzkinSOS}. The following decomposition (found by solving an LP) shows that $M\in 2DSOS_{3,6}$ and thus that $M$ is nonnegative:
\begin{align}
M(x)\cdot\Big{(}\sum_{i=1}^3 x_i^2\Big{)}^2 = &\frac{1}{2}(x_3^5-x_3^3 x_2^2)^2 + 
         \frac{1}{2}(x_2^4 x_1-x_3^4 x_1)^2 + 
         \frac{1}{2}(x_2^4 x_1-x_3^2 x_2^2 x_1)^2 \nonumber \\
         +&\frac{1}{2}(x_3^5-x_3^3 x_1^2)^2 + 
         \frac{1}{2}(x_3^3 x_1^2-x_3^3 x_2^2)^2 + \nonumber 
         \frac{5}{2}(x_2^3 x_1^2-x_3^4 x_2)^2 \\
         +&\frac{1}{2}(x_2^3 x_1^2-x_3^2 x_2 x_1^2)^2 + 
         \frac{5}{2}(x_2^2 x_1^3-x_3^4 x_1)^2 + \nonumber
         \frac{1}{2}(x_2^2 x_1^3-x_3^2 x_2^2 x_1)^2  \\
         +&\frac{1}{2}(x_2 x_1^4-x_3^4 x_2)^2 + 
         \frac{1}{2}(x_2 x_1^4-x_3^2 x_2 x_1^2)^2. \nonumber
\end{align}	
This proves that $2DSOS_{3,6}\nsubseteq SOS_{3,6}$. 	}
\end{Example}

\ani{
\begin{Example}
	Consider the polynomial $p(x)=x_1^4x_2^2+x_2^4x_3^2+x_3^4x_1^2-3x_1^2x_2^2x_3^2$. Once again this polynomial is nonnegative but \emph{not} a sum of squares~\cite{Reznick}. The following decomposition, which was obtained by solving an LP, shows that $p\in 1DSOS_{3,6}$ and thus that $p$ is nonnegative:
\begin{align}
p(x)\cdot\Big{(}\sum_{i=1}^3 x_i^2\Big{)} = &(x_3 x_2^2 x_1-x_3^3 x_1)^2 + 
         \frac{1}{2}(x_3^2 x_1^2-x_3^2 x_2^2)^2 + 
         (x_3 x_2^3-x_3 x_2 x_1^2)^2 \nonumber   \\
         +&\frac{1}{2}(x_2^2 x_1^2-x_3^2 x_2^2)^2 + 
         \frac{1}{2}(x_2^2 x_1^2-x_3^2 x_1^2)^2 + 
         (x_2 x_1^3-x_3^2 x_2 x_1)^2. \nonumber
\end{align}	
This proves that $1DSOS_{3,6}\nsubseteq SOS_{3,6}$. 
\end{Example}
}

It is natural to ask whether every nonnegative polynomial is r-dsos (or r-sdsos) for some $r$? The following theorems (Theorem~\ref{thm:even.forms.rdsos} and \ref{thm:p.dsos.Polya}) deal with this question and provide asymptotic guarantees on r-dsos (and hence r-sdsos) hierarchies. 



\begin{theorem}\label{thm:even.forms.rdsos}
Let $p$ be an even\footnote{An even polynomial \aaa{(resp. monomial)} is a polynomial \aaa{(resp. monomial)} whose individual variables are raised only to even degrees.} positive definite form. Then, there exists an integer $r$ for which $p$ is r-dsos (and hence r-sdsos). 
\end{theorem} 

\begin{proof}
A well-known theorem of P\'olya~\cite{polya_simplex} states that if a form $f(x_1,\ldots,x_n)$ is positive on the simplex (i.e., the set $\Delta_n\mathrel{\mathop:}=\{x\in\mathbb{R}^n| x_i\geq 0, \sum_i x_i=1\}$), then there exists an integer $r$ such that all coefficients of $$(x_1+\cdots+x_n)^rf(x_1,\ldots,x_n)$$
are positive. Under the assumptions of the theorem, this implies that there is an $r$ for which the form $p(\sqrt{x_1},\ldots, \sqrt{x_n})(x_1+\cdots+x_n)^r$ and hence the form $$p(x_1,\ldots,x_n)(x_1^2+\cdots+x_n^2)$$  have positive coefficients. But this means that $p(x)(\sum_i x_i^2)$ is a nonnegative weighted sum of squared monomials and therefore clearly dsos (see (\ref{eq:dsos})). (In a Gram matrix representation $p(x)(\sum_i x_i^2)=z^T(x)Qz(x)$, the argument we just gave shows that we can take $Q$ to be diagonal.)
\end{proof}



We remark that even forms already constitute a very interesting class of polynomials since they include, e.g., all polynomials coming from \emph{copositive programming}; see Subsection~\ref{sec:copositivity}. In fact, any NP-complete problem can be reduced in polynomial time to the problem of checking whether a degree-4 even form is nonnegative~\cite{nonnegativity_NP_hard}. An application of this idea to the independent set problem in combinatorial optimization is presented in Subsection~\ref{sec:copositivity}. 



The next proposition shows that the evenness assumption cannot be removed from Theorem~\ref{thm:even.forms.rdsos}.


\begin{proposition}\label{prop.duh.cedric}
For any $0<a<1$, the quadratic form
\begin{equation} \label{eq:not.rsdsos.poly}
 p(x_1,x_2,x_3)=(x_1+x_2+x_3)^2+a(x_1^2+x_2^2+x_3^2)
\end{equation}
is positive definite but not r-sdsos for any $r$.
\end{proposition}

\begin{proof}
\aaa{Positive definiteness is immediate from having $a>0$. We show that for $a<1,$ the quadratic form $p$ is not r-sdsos for any $r$ by first presenting a general separating hyperplane for the cone of sdsos forms in any degree and dimension. For a form $f$, construct a vector $v_f$ which has an entry per monomial in $f$, with entry $i$ equal to $+1$ if the $i$-th monomial in $f$ is even and $-1$ if it is not. We claim that if $f$ is sdsos, then the (standard) inner product of its coefficients with $v_f$ is nonnegative. This can be seen by noting that this inner product is equal to the value that a related sdsos form $q$ takes at the all ones vector. The form $q$ is obtained from $f$ by leaving its even monomials untouched and flipping the sign of the coefficients of its non-even monomials. It is straightforward to see (from the expansion in the definition of an sdsos polynomial in (\ref{eq:sdsos})) that $q$ constructed as such will be sdsos as well. Hence, $q(1,\ldots,1)\geq 0$.}

\aaa{The latter claim of the proposition now follows from the observation that for any $r$, the inner product of the coefficients of the form $f_r(x)\mathrel{\mathop:}=p(x)\cdot(x_1^2+x_2^2+x_3^2)^r$ with the $\pm 1$ vector $v_{f_r}$  that we described above is negative when $a<1$ (and exactly when $a<1$).}
%
\end{proof}

We remark that the form in (\ref{eq:not.rsdsos.poly}) can be proved to be positive \aaa{using} the improvements presented in Section~\ref{sec:improvements}, in particular with a ``factor width 3'' proof system from Subsection~\ref{subsec:factor.width}. {\color{black} Alternatively, one can work with the next theorem, which removes the assumption of evenness from Theorem~\ref{thm:even.forms.rdsos} at the cost of doubling the number of variables and the degree.


\begin{theorem} [see Section 4 of~\cite{pop_hierarchy}]\label{thm:p.dsos.Polya}
	An $n$-variate form $p\mathop{\mathrel:}=p(x)$ of degree $2d$ is positive definite if and only if there exists a positive integer $r$ that makes the following form 	$r$-dsos:
	$$p(v^2-w^2)-\frac{1}{\sqrt{r}}(\sum_{i=1}^n (v_i^2-w_i^2)^2)^d+\frac{1}{2\sqrt{r}}(\sum_{i=1}^n(v_i^4+w_i^4))^d.$$
\end{theorem}

In~\cite[Section 4]{pop_hierarchy}, this theorem is used to construct LP and SOCP-based converging hierarchies of lower bounds\footnote{As implied by Proposition~\ref{prop.duh.cedric}, not every convergent SOS hierarchy has an analogous convergent (S)DSOS hierarchy; see~\cite[Section 2.2]{ahmadi2017response} for more detail.}for the general polynomial optimization problem (\ref{eq:POP}) when the feasible set is compact. 
}

\section{Numerical examples and applications} \label{sec:experiments}

In this section, we consider several numerical examples that highlight the scalability of the approach presented in this paper. Comparisons to examples considered in the literature are provided whenever possible. A software package written using the Systems Polynomial Optimization Toolbox (SPOT) \cite{Megretski_spot} includes a complete implementation of the presented methods and is available online\footnote{Link to \texttt{spotless\_isos} software package: 


\href{https://github.com/anirudhamajumdar/spotless/tree/spotless_isos}{https://github.com/anirudhamajumdar/spotless/tree/spotless\_isos}}. The toolbox features efficient polynomial algebra and allows us to setup the large-scale LPs and SOCPs arising from our examples. Our Supplementary Material provides a brief introduction to the software package and the associated functionality required for setting up the DSOS/SDSOS programs considered in this paper. \ani{The LPs and SOCPs resulting from our programs, as well as the SDPs that are considered for comparison, are solved using the solver MOSEK~\cite{mosek} on a machine with 4 processors and a clock speed of 3.4 GHz and 16GB of RAM. While using special-purpose solvers could lead to speedups for any of the three methods (LP/SOCP/SDP) for specific problems, we believe that a widely-employed general-purpose solver such as MOSEK provides a strong benchmark. Further, for our example on options pricing (Section \ref{subsec:experiments.options.pricing}), we also present results using the SDPNAL+ solver \cite{yang2015sdpnal}. As brought to our attention by a referee, the constraints in this example are particularly amenable to the semismooth Newton method employed by SDPNAL+.}  In our experiments, we have also tried the solver SeDuMi~\cite{Sturm99}, which is arguably the most commonly-used free SDP solver. We do not report these run times however (except in Section~\ref{subsub:ROA}) as they are often significantly slower than those produced by the SDP solver of MOSEK.








\subsection{Lower bounds on polynomial optimization problems} \label{subsec:experiments.pop}

An important application of sum of squares optimization is to obtain lower bounds on polynomial optimization problems (see e.g.~\cite{lasserre_moment},~\cite{Minimize_poly_Pablo},~\cite{nie2006minimizing}). In this subsection, we consider the particular problem of minimizing a homogeneous polynomial $p(x)$ of degree $2d$ on the unit sphere (i.e., the set $\{x\in\mathbb{R}^n| \ x^Tx=1\}$). This well-studied problem is strongly NP-hard even when $d=2$. Its optimal value is easily seen to be equivalent to the optimal value of the following problem:
\begin{flalign} 
\label{eq:min_hom_sphere}
			\underset{\gamma}{\text{maximize}} \hspace*{1cm} & \gamma  \\
			\text{s.t.} \hspace*{1cm} & p(x) - \gamma (x^T x)^d \geq 0, \forall x\in\mathbb{R}^n. \nonumber
\end{flalign}

By replacing the nonnegativity condition with a SOS/DSOS/SDSOS constraint, the resulting problem becomes an SDP/LP/SOCP. Tables \ref{tab:min_hom_opts} and \ref{tab:min_hom_runtimes} compare optimal values and running times respectively on problems of increasing size. We restrict ourselves to quartic forms and vary the number of variables $n$ between $5$ and $70$. Table \ref{tab:min_hom_opts} compares the optimal values for the DSOS, SDSOS, 1-DSOS, 1-SDSOS and SOS relaxations of problem \eqref{eq:min_hom_sphere} on quartic forms whose coefficients are fully dense and drawn independently from the standard normal distribution (we consider a single random instance for each $n$). We note that the optimal values presented here are representative of numerical experiments we performed for a broader set of random instances (including coefficients drawn from a uniform distribution instead of the normal distribution). Also, in all our experiments, we could verify that the SOS bound actually coincides with the true global optimal value of the problem. This was done by finding a point on the sphere that produced a matching upper bound.

As the tables illustrate, both DSOS and SDSOS scale up to problems of dimension $n = 70$. This is well beyond the size of problems handled by SOS programming, which is unable to deal with problems with $n$ larger than $25$ due to memory (RAM) constraints. While the price we pay for this increased scalability is suboptimality, this is tolerable in many applications (as we show in other examples in this section).

Table \ref{tab:min_hom_opts} also compares the approach presented in this paper to the lower bound obtained at the root node of the branch-and-bound procedure in the popular global optimization package BARON \cite{Sahinidis14}.\footnote{We thank Aida Khajavirad for running the BARON experiments for us.} This comparison was only done up to dimension 30 since beyond this we ran into difficulties passing the large polynomials to BARON's user interface.

\begin{table}
\footnotesize
\tabcolsep=0.12cm
\begin{center}
  \begin{tabular}{ | c | c | c | c | c | c | c | c | c | c | c |}
    \hline 
     & n = 10 & n = 15 & n = 20 & n = 25 & n = 30 & n = 40 & n = 50 & n = 60 & n = 70 \\ \hline
    DSOS & -5.31 & -10.96 & -18.012 &  -26.45 &  -36.85 &  -62.30 &  -94.26 & -133.02 & -178.23  \\ \hline
    SDSOS &  -5.05 &  -10.43 &  -17.33 &  -25.79 & -36.04 &  -61.25 &  -93.22 & -131.64 & -176.89  \\ \hline
    1-DSOS & -4.96  &  -9.22 &  -15.72 &  -23.58 & NA & NA & NA & NA & NA  \\ \hline
    1-SDSOS & -4.21  & -8.97 & -15.29 & -23.14 & NA & NA & NA & NA & NA  \\ \hline
    SOS & -1.92 & -3.26 & -3.58 & -3.71 & NA & NA & NA & NA & NA  \\ \hline
    BARON & -175.41 & -1079.89 & -5287.88 & - & -28546.1 & - & - & - & -  \\ 
    \hline
  \end{tabular}
    \caption{Comparison of lower bounds on the minimum of a quartic form on the sphere for varying number of variables.}
  \label{tab:min_hom_opts}
  \end{center}
\end{table}


Table \ref{tab:min_hom_runtimes} presents running times\footnote{The BARON instances were implemented on a 64-bit Intel Xeon X5650 2.66Ghz processor using the CPLEX LP solver.} averaged over $10$ random instances for each $n$. There is a significant difference in running times between DSOS/SDSOS and SOS beginning at $n = 15$. While we are unable to run SOS programs beyond $n = 25$, DSOS and SDSOS programs for $n = 70$ take only a few minutes to execute. Memory constraints prevent us from executing programs for $n$ larger than $70$. The results indicate, however, that it may be possible to run larger programs within a reasonable amount of time on a machine equipped with more RAM.

\begin{table}
\footnotesize
\tabcolsep=0.12cm
\begin{center}
  \begin{tabular}{ | c | c | c | c | c | c | c | c | c | c | c |}
    \hline 
     & n = 10 & n = 15 & n = 20 & n = 25 & n = 30 & n = 40 & n = 50 & n = 60 & n = 70 \\ \hline
    DSOS & 0.30 & 0.38 & 0.74 & 15.51 & 7.88 & 10.68 & 25.99 & 58.10 & 342.76  \\ \hline
    SDSOS &  0.27 & 0.53 & 1.06 & 8.72 & 5.65 &  18.66 & 47.90 & 109.54 & 295.30  \\ \hline
    1-DSOS & 0.92 & 6.26 & 37.98 & 369.08 & $\infty$ & $\infty$ & $\infty$ & $\infty$ & $\infty$  \\ \hline
    1-SDSOS &  1.53 & 14.39 & 82.30 & 538.5389 & $\infty$ & $\infty$ & $\infty$ & $\infty$ & $\infty$  \\ \hline
    SOS & 0.24 & 5.60 & 82.22 & 1068.66 & $\infty$ & $\infty$ & $\infty$ & $\infty$ & $\infty$  \\ \hline
    BARON & 0.35 & 0.62 & 3.69 & - & -  & - & - & - & -  \\ 
    \hline
  \end{tabular}
    \caption{Comparison of running times (in seconds) averaged over $10$ instances for lower bounding a quartic form on the sphere for varying number of variables.}
\label{tab:min_hom_runtimes}
  \end{center}
\end{table}

\subsection{Copositive programming and combinatorial optimization}
\label{sec:copositivity}

A symmetric matrix $Q \in \mathbb{R}^{n \times n}$ is copositive if $x^T Q x\geq 0$ for all $x\geq 0$ (i.e., for all $x$ in the nonnegative orthant). The problem of optimizing a linear function over affine sections of the cone $\mathcal{C}_n$ of $n \times n$ copositive matrices has recently received a lot of attention from the optimization community~\cite{dur_copositive}, since it can exactly model several problems of combinatorial and nonconvex optimization~\cite{Bomze02},~\cite{dur_copositive},~\cite{burer_copositive}. For instance, the size $\alpha(G)$ of the largest independent set \footnote{An independent set of a graph is a subset of its nodes no two of which are connected. The problem of finding upper bounds on $\alpha$ has many applications in scheduling and coding theory~\cite{lovasz_shannon}.} of a graph $G$ on $n$ nodes is equal to the optimal value of the following copositive program~\cite{deKlerk_StableSet_copositive}:
\begin{flalign} \label{eq:copositive.stable.set}
\underset{\gamma}{\text{minimize}} \hspace*{1cm} & \gamma \nonumber \\
\text{s.t.} \hspace*{1cm} &  \gamma(A+I) -  J \in \mathcal{C}_n, 
\end{flalign}
where $A$ is the adjacency matrix of the graph, $I$ is the identity matrix, and $J$ is the matrix of all ones.


 It is easy to see that a matrix $Q$ is copositive if and only if the quartic form $(x.^2)^TQ(x.^2)$, with $x.^2\mathrel{\mathop:}= (x_1^2, x_2^2, \dots, x_n^2)^T$, is nonnegative. Hence, by requiring this form to be r-sos/r-sdsos/r-dsos, one obtains inner approximations to the cone $\mathcal{C}_n$ that can be optimized over using semidefinite, second order cone, and linear programming. The sos version of this idea goes back to the PhD thesis of Parrilo~\cite{PhD:Parrilo}. 
 
In this section, we compare the quality of these bounds using another hierarchical LP-based method in the literature \cite{Bomze02},~\cite{deKlerk_StableSet_copositive}. The $r$-th level of this LP hierarchy (referred to as the P\'olya LP from here onwards) requires that 
 the coefficients of the form $(x.^2)^TQ(x.^2)\cdot(x_1^2+\cdots+x_n^2)^r$ be nonnegative. This is motivated by a theorem by P\'olya, which states that this constraint will be satisfied for $r$ large enough (assuming $x^TQx>0$ for all nonzero $x\geq 0$). It is not difficult to see that the P\'olya LP is always dominated by the LP which requires $(x.^2)^TQ(x.^2)$ to be r-dsos. This is because the latter requires the underlying Gram matrix to be diagonally dominant, while the former requires it to be diagonal, with nonnegative diagonal entries. 
 
In~\cite{deKlerk_StableSet_copositive}, de Klerk and Pasechnik show that when the P\'olya LP is applied to approximate the copositive program in (\ref{eq:copositive.stable.set}), then the precise independent set number of the graph is obtained within $\alpha(G)^2$ steps of the hierarchy. By the previous argument, the same claim immediately holds for the r-dsos (and hence r-sdsos) hierarchies. 
 
Table~\ref{tab:indep.set.icos} revisits Example 5.2 of \cite{Bomze02}, where upper bounds on the independent set number of the complement of the graph of an icosahedron are computed. (This is a graph on 12 nodes with independent set number equal to 3; see~\cite{Bomze02} for its adjacency matrix.)  As the results illustrate, SOS programming provides an exact upper bound (since the size of the largest stable set is necessarily an integer). The second levels of the r-dsos and r-sdsos hierarchies also provide an exact upper bound. In contrast, the LP based on Polya's theorem in \cite{Bomze02} gives the upper bound $\infty$ for the $0$-th and the $1$-st level, and an upper bound of $6$ at the second level. In contrast to the P\'olya LP, one can show that the $0$-th level of the r-dsos LP always produces a finite upper bound on the independent set number. In fact, this bound will always be smaller than $n-\min d_i+1$, where $d_i$ is the degree of node $i$~\cite[Theorem 5.1]{col_gen_Ahmadi_Dash_Hall} .

\begin{table}\label{tab:indep.set.icos}
\small
\begin{center}
  \begin{tabular}{ | c | c | c | c | c |  }
    \hline 
     & Polya LP & r-DSOS & r-SDSOS & SOS \\ \hline
    $r = 0$ & $\infty$ & 6.000 & 6.000 & 3.2362  \\ \hline
    $r = 1$ & $\infty$ & 4.333 & 4.333 & NA \\ \hline
    $r = 2$ & 6.000 & 3.8049 & 3.6964 & NA \\
    \hline
  \end{tabular}
    \caption{Comparison of upper bounds on the size of the largest independent set in the complement of the graph of an icosahedron.}
\label{tab:stable_set_results}
  \end{center}
  \end{table}

\subsection{Convex regression in statistics} \label{subsec:experiments.convex.regression}

Next, we consider the problem of fitting a function to data subject to a constraint on the function's convexity. This is an important problem in statistics and has a wide domain of applications including value function approximation in reinforcement learning, options pricing and modeling of circuits for geometric programming based circuit design \cite{hannah_convex_regression,Hannah12}. Formally, we are given $N$ pairs of data points $(x_i,y_i)$ with $x_i \in \mathbb{R}^n$, $y_i \in \mathbb{R}$ and our task is to find a \emph{convex} function $f$ from a family $\mathcal{F}$ that minimizes an appropriate notion of fitting error (e.g., $L_1$ error):
\begin{flalign} 
					\underset{f \in \mathcal{F}}{\text{minimize}} \hspace*{1cm} & \sum_{i=1}^N |f(x_i) - y_i| \nonumber \\
			\text{s.t.} \hspace*{1cm} & f \ \text{is convex}. \nonumber
\end{flalign}
The convexity constraint here can be imposed by requiring that the function $w^T H(x)w$ in $2n$ variables $(x,w)$ associated with the Hessian matrix $H(x)$ of $f(x)$ be nonnegative. Restricting ourselves to polynomial functions $f$ of bounded degree, we obtain the following optimization problem:
\begin{flalign} 
			\underset{f \in \mathbb{R}^d[x]}{\text{minimize}} \hspace*{1cm} & \sum_{i=1}^N |f(x_i) - y_i| \nonumber \\
			\text{s.t.} \hspace*{1cm} & w^T H(x)w \geq 0, \forall x,w,\nonumber
\end{flalign}
where $H(x)$ is again the Hessian of $f(x)$. As before, we can replace the nonnegativity constraint with a dsos/sdsos/sos constraint. For our numerical experiment, we generated $300$ random vectors $x_i$ in $\mathbb{R}^{20}$ drawn i.i.d. from the standard normal distribution. The function values $y_i$ were computed as follows:
\begin{equation}
\ani{y_i = \text{exp}(\|x_i\|_2) + \eta_i}, \nonumber
\end{equation}
where \ani{$\eta_i$} was chosen i.i.d. from the standard normal distribution. Tables \ref{tab:convex_reg_errors} and \ref{tab:convex_reg_runtimes} present the fitting errors and running times for the DSOS/SDSOS/SOS programs resulting from restricting the class of functions $\mathcal{F}$ to polynomials of degree $d = 2$ and $d = 4$. As the results illustrate, we are able to obtain significantly smaller errors with polynomials of degree $4$ using DSOS and SDSOS (compared to SOS with $d=2$), while the SOS program for $d=4$ does not run due to memory constraints. 

\begin{table}
\small
\begin{minipage}{0.5\textwidth}
\begin{center}
  \begin{tabular}{ | c | c | c | c | c |  }
    \hline 
     & DSOS & SDSOS & SOS \\ \hline
    $d = 2$ & 35.11 & 33.92 & 21.28   \\ \hline
    $d = 4$ & 14.86 & 12.94 & NA   \\ 
    \hline
  \end{tabular}
    \caption{Comparison of fitting errors for convex regression.}
\label{tab:convex_reg_errors}

  \end{center}
  \end{minipage}
\begin{minipage}{0.49\textwidth}
\begin{center}
  \begin{tabular}{ | c | c | c | c | c |  }
    \hline 
     & DSOS & SDSOS & SOS \\ \hline
    $d = 2$ & 0.49 & 0.66 & 0.53 \\ \hline
    $d = 4$ & 145.85 & 240.18 & $\infty$   \\ 
    \hline
  \end{tabular}
    \caption{Comparison of run times (in s) for convex regression.}
\label{tab:convex_reg_runtimes}
  \end{center}
\end{minipage}
  \end{table}

%

\subsection{Options pricing}\label{subsec:experiments.options.pricing}
An important problem in financial economics is that of determining the price of a derivative security such as an option given the price of the underlying asset. Significant progress was made in tackling this problem with the advent of the Black-Scholes formula, which operates under two assumptions: (i) the underlying stock price is governed by geometric Brownian motion, and (ii) there is no arbitrage. A natural question that has been considered in financial mathematics is to see if one can provide bounds on the price of an option in the no-arbitrage setting given much more minimal assumptions on the dynamics of the stock price. One approach to this question---studied, e.g., in~\cite{Bertsimas02},~\cite{Boyle_options}, ---is to assume that we are only given the first $k$ moments of the stock price and want to optimally bound the price of the option.

More precisely, we consider that we are given $m$ stocks, an associated option with payoff function $\phi: \mathbb{R}_+^m \rightarrow \mathbb{R}$ (which will typically depend on the strike price of the option), a vector of $n$ moment functions, $f_i: \mathbb{R}_m^+ \rightarrow \mathbb{R}$, $i = 0,1,\dots,n$ ($f_0$ is assumed to be 1), and the corresponding vector of moments $q = (q_0,\dots,q_n)$ (here $q_0 = 1$). The problem of upper bounding the price of the option can then be formulated as follows \cite{Bertsimas02}:
\begin{flalign} 
\label{eq:options_pricing_orig}
			\underset{\pi}{\text{maximize}} \hspace*{1cm} & E_\pi[\phi(X)]  \\
			\text{s.t.} \hspace*{1cm} & E_\pi[f_i(X)] = q_i, \ i = 0,\dots,n, \nonumber \\
			& \pi(x) \geq 0, \forall x \in \mathbb{R}_+^m. \nonumber
\end{flalign}
Here, the expectation is taken over all Martingale measures $\pi$, defined on $\mathbb{R}_+^m$. We consider here the case where one is given the mean $\mu$ and covariance matrix $\sigma$ of the stock prices. Problem \eqref{eq:options_pricing_orig} can then be cast as the following problem (see \cite[Section 6.2]{Bertsimas02})
\begin{flalign}
 \label{eq:options_pricing_copos}
			\underset{y,Y}{\text{minimize}} \hspace*{1cm} & y_0 + \sum_{i=1}^n y_i \mu_i + \sum_{i=1}^n \sum_{j=1}^n y_{ij}(\sigma_{ij} + \mu_i\mu_j)  \\
			\text{s.t.} \hspace*{1cm} & x^T Y x + y^T x + y_0 \geq \phi(x), \ \forall x \geq 0, \nonumber
\end{flalign}
which for many common functions $\phi$ of interest (see e.g. below) gives rise to a copositive program. A well-known and obvious sufficient condition for a matrix $M$ to be copositive is for it to have a decomposition $M = P + N$, where $P$ is psd and $N$ is element-wise nonnegative \cite{PhD:Parrilo}. This allows one to apply SDP to problem (\ref{eq:options_pricing_copos}) and obtain upper bounds on the option price. By replacing the psd condition on the matrix $P$ with a dd/sdd condition, we obtain a DDP/SDDP.

We first compare the DDP/SDDP/SDP approaches on an example from \cite[Section 4]{Boyle_options}, which considers the problem of upper bounding the price of a European call option with $m = 3$ underlying assets. The vector of means of the assets is $\mu = [44.21, 44.21, 44.21]^T$ and the covariance matrix is:
\begin{equation}
\tabcolsep=1pt
\setcounter{MaxMatrixCols}{12}
\sigma = \begin{bmatrix} 
  184.04 & 164.88 & 164.88  \\
  164.88 & 184.04 & 164.88  \\
  164.88 & 164.88 & 184.04
   \end{bmatrix}.
\end{equation}

Further, we have the following payoff function, which depends on the strike price $K$:
\begin{equation}
\phi(x) = \textrm{max}(x_1 - K, x_2 - K, x_3 - K, 0).
\end{equation}

\noindent Table \ref{tab:options_pricing_opts} compares the upper bounds for different strike prices using DDP/SDDP/SDP. \ani{For the SDP experiments, we present results using both MOSEK and SDPNAL+ \cite{yang2015sdpnal}. }
For each strike price, the upper bound obtained using SDP \aaa{(solved with MOSEK)} is exact. This can be verified by finding a distribution that achieves the upper bound (i.e., finding a matching lower bound). For example, when $K = 30$, the distribution supported on the set of four points $(5.971, 5.971, 5.971)$, $(54.03, 46.02, 46.02, 46.02)$, $(46.02, 54.03, 46.02)$ and $(46.02, 46.02, 54.03)$ in $\mathbb{R}^3$ with probability masses $0.105$, $0.298$, $0.298$, $0.298$ respectively does the job. \ani{As the table illustrates, the upper bounds obtained using SDPNAL+ are slightly numerically inaccurate for the default accuracy tolerance of $10^{-6}$, but are accurate with a tolerance of $10^{-9}$ (this was the largest tolerance that resulted in bounds accurate to two decimal places).}
The upper bound obtained using SDDP is almost identical to the exact bound for each strike price. The DDP bound is loose, and interestingly does not change with the strike price. Running times for the different methods on this small example are negligible and hence not presented.


\begin{table}
\small
\begin{center}
  \begin{tabular}{ | c | c | c | c | c | c | c | c | }
    \hline 
             & Exact & SDP  & \ani{SDP} & \ani{SDP} & SDDP & DDP \\
             & & MOSEK & \ani{SDPNAL+ (tol=$10^{-6}$)} & \ani{SDPNAL+ (tol=$10^{-9}$)}  & & \\ \hline
    $K = 30$ & 21.51 & 21.51 & \ani{21.72} & \ani{21.51} & 21.51 & 132.63 \\ \hline
    $K = 35$ & 17.17 & 17.17 & \ani{18.07} & \ani{17.17} &  17.17 & 132.63 \\ \hline
    $K = 40$ & 13.20  & 13.20 & \ani{12.37} & \ani{13.20} & 13.20 & 132.63 \\ \hline
    $K = 45$ & 9.84  & 9.84 & \ani{9.66} & \ani{9.84} & 9.85 & 132.63 \\ \hline
    $K = 50$ & 7.30  & 7.30 & \ani{7.90} & \ani{7.30} & 7.30 & 132.63 \\ \hline
  \end{tabular}
    \caption{Comparison of upper bounds on options prices for different strike prices obtained using SDP, SDDP and DDP. Here, the number of underlying assets is $m = 3$. \ani{For the SDP experiments, we present results using both MOSEK and SDPNAL+ (with two different tolerance levels).}}
\label{tab:options_pricing_opts}
  \end{center}
  \end{table}
  
In order to demonstrate the scalability of DDP and SDDP, we consider a larger scale randomized example with $m = 50$ underlying assets. The mean and covariance are taken to be the sample mean and covariance of $100$ instances of vectors belonging to $\mathbb{R}^m$ with elements drawn uniformly and independently from the interval $[0,10]$. The strike price is $K = 5$ and the payoff function is again
\begin{equation}
\phi(x) = \textrm{max}(x_1 - K, x_2 - K, \dots, x_m - K, 0).
\end{equation}

Table \ref{tab:options_pricing_large} compares upper bounds and running times for the different methods. SDDP provides us with a bound that is very close to the one obtained using SDP, with a significant speedup in computation \aaa{(approximately a factor of 100 compared to MOSEK and a factor of 90 compared to SDPNAL+)}. The running time for DDP is comparable to SDDP, but the bound is not as tight. \ani{SDPNAL+ (tolerance = $10^{-6}$) is faster as compared to MOSEK, but yields a slightly larger bound. In contrast to the smaller-scale example above, the bound does not change when the SDPNAL+ tolerance is set to $10^{-9}$.}

\begin{table} 
\small
\begin{center}
  \begin{tabular}{ | c | c | c | c | c | c |}
    \hline 
              & SDP & \ani{SDP} & \ani{SDP} & SDDP & DDP \\ 
              & MOSEK & \ani{SDPNAL+ (tol=$10^{-6}$)} & \ani{SDPNAL+ (tol=$10^{-9}$)}  & & \\ \hline
    Upper bound & 18.76 & \ani{18.88} & \ani{18.88} & 19.42 & 252.24  \\ \hline
    Running time & 2502.6 s & \ani{2117.3 s} & \ani{6221.7 s} & 24.36 s & 11.85 s  \\ \hline
  \end{tabular}
    \caption{Comparison of upper bounds and running times for a large ($m = 50$) options pricing example using SDP, SDDP and DDP. \ani{For the SDP experiments, we present results using both MOSEK and SDPNAL+ (with two different tolerance levels).}}
\label{tab:options_pricing_large}
  \end{center}
  \end{table}

\subsection{Sparse PCA}\label{subsec:experiments.sparse.pca}

Next, we consider the problem of \emph{sparse principal component analysis} (sparse PCA). In contrast to standard PCA, where the principal components (PCs) in general depend on \emph{all} the observed variables, the goal of sparse PCA is to identify principal components that only depend on small subsets of the variables \cite{Zou06}. While the statistical fidelity of the resulting representation of the data in terms of sparse PCs will in general be lower than the standard PCs, sparse PCs can significantly enhance interpretability of the results. This feature has proved to be very useful in applications such as finance and analysis of gene expressions; see, e.g., \cite{dAspremont07} and references therein. 

Given a $n \times n$ covariance matrix $A$, the problem of finding sparse principal components can be written as the following optimization problem:
\begin{flalign} 
\label{eq:sparse_pca_orig}
				\underset{x}{\text{maximize}} \hspace*{1cm} & x^T A x  \\
			\text{s.t.} \hspace*{1cm} & \|x\|_2 = 1, \nonumber \\
			& \textrm{Card}(x) \leq k. \nonumber
\end{flalign}
Here, $\textrm{Card}(x)$ is the cardinality (number of non-zero entries) of $x$ and $k$ is a given threshold. The problem can be reformulated as the following rank-constrained matrix optimization problem~\cite{dAspremont07}:
\begin{flalign} 
\label{eq:sparse_pca_rank}
				\underset{X}{\text{maximize}} \hspace*{1cm} & \textrm{Tr}(AX)  \\
			\text{s.t.} \hspace*{1cm} & \textrm{Tr}(X) = 1, \nonumber \\
			& \textrm{Card}(X) \leq k^2, \nonumber \\
                        & X \succeq 0, \ \textrm{Rank}(X) = 1. \nonumber
\end{flalign}
In \cite{dAspremont07}, the authors propose ``DSPCA'', an SDP relaxation of this problem that is obtained by dropping the rank constraint and replacing the cardinality constraint with a constraint on the $l_1$ norm:
\begin{flalign}
\label{eq:sparse_pca_sdp}
			\underset{X}{\text{maximize}} \hspace*{1cm} & \textrm{Tr}(AX)  \\
			\text{s.t.} \hspace*{1cm} & \textrm{Tr}(X) = 1, \nonumber \\
			& {\bf{1}}^T |X| {\bf{1}} \leq k, \nonumber \\
                        & X \succeq 0. \nonumber
\end{flalign}
The optimal value of problem \eqref{eq:sparse_pca_sdp} is an upper bound on the optimal value of \eqref{eq:sparse_pca_orig}. Further, when the solution $X_1$ to \eqref{eq:sparse_pca_sdp} has rank equal to one, the SDP relaxation is \emph{tight} and the dominant eigenvector $x_1$ of $X_1$ is the optimal loading of the first sparse PC \cite{dAspremont07}. When a rank one solution is not obtained, the dominant eigenvector can still be retained as an approximate solution to the problem. Further sparse PCs can be obtained by deflating $A$ to obtain:
$$A_2 = A - (x_1^T A x_1) x_1 x_1^T$$
and re-solving the SDP with $A_2$ in place of $A$ (and iterating this procedure). 

The framework presented in this paper can be used to obtain LP and SOCP relaxations of \eqref{eq:sparse_pca_orig} by replacing the constraint $X \succeq 0$ by the constraint $X \in DD_n^*$ or $X \in SDD_n^*$ respectively, where $DD_n^*$ and $SDD_n^*$ are the dual cones of $DD_n$ and $SDD_n$ (see~\cite[Section 3.3]{basis_pursuit_Ahmadi_Hall} for a description of these dual cones). Since $S_n^+ \subseteq SDD_n^* \subseteq DD_n^*$, we are guaranteed that the resulting optimal solutions will be upper bounds on the SDP solution. Further, as in the SDP case, the relaxation is tight when a rank-one solution is obtained, and once again, if a rank-one solution is not obtained, we can still use the dominant eigenvector as an approximate solution.

We first consider Example $6.1$ from \cite{dAspremont07}. In this example, there are three hidden variables distributed normally:
$$ V_1 \sim \mathcal{N}(0,290), \ V_2 \sim \mathcal{N}(0,300), \ V_3 = -0.3 V_1 + 0.925 V_2 + \epsilon, \ \epsilon \sim \mathcal{N}(0,1).$$
These hidden variables generate $10$ observed variables:
$$X_i = V_j + \epsilon_i^j, \epsilon_i^j \sim \mathcal{N}(0,1),$$
with $j = 1$ for $i = 1, \dots, 4$, $j = 2$ for $i = 5, \dots, 8$ and $j = 3$ for $i = 9,10,$ and $\epsilon_i^j$ independent for $j = 1,2,3$ and $i = 1,\dots,10$. This knowledge of the distributions of the hidden and observed variables allows us to compute the exact $10 \times 10$ covariance matrix for the observed variables. The sparse PCA algorithm described above can then be applied to this covariance matrix. Table \ref{tab:sparse_pca_small_comparisons} presents the first two principal components computed using standard PCA, DDP, SDDP and DSPCA (corresponding to SDP). As the table illustrates, the loadings corresponding to the first two PCs computed using standard PCA are not sparse. All other methods (DDP, SDDP, and SDP) give \emph{exactly} the same answer and correspond to a rank-one solution, i.e., the optimal sparse solution. As expected, the sparsity comes at the cost of a reduction in the variance explained by the PCs (see \cite{Zou06} for computation of the explained variance).

\begin{table}
\footnotesize
\tabcolsep=0.10cm
\begin{center}
  \begin{tabular}{ | c | c | c | c | c | c | c | c | c | c | c | c | c | }
    \hline 
     & $X_1$ & $X_2$ & $X_3$ & $X_4$ & $X_5$ & $X_6$ & $X_7$ & $X_8$ & $X_9$ & $X_{10}$ & expl. var. \\ \hline
    PCA, PC1 & 0.116 & 0.116 & 0.116 & 0.116 & -0.395 & -0.395 & -0.395 & -0.395 & -0.401 & -0.401 & 60.0 \%  \\ 
    PCA, PC2 & -0.478 & -0.478 & -0.478 & -0.478 & -0.145 & -0.145 & -0.145 & -0.145 & 0.010 & 0.010 & 39.6 \%  \\ \hline 
    Other, PC1 & 0 & 0 & 0 & 0 & -0.5 & -0.5 & -0.5 & -0.5 & 0 & 0 & 40.9 \%  \\
    Other, PC2 & 0.5 & 0.5 & 0.5 & 0.5 & 0 & 0 & 0 & 0 & 0 & 0 & 39.5 \%  \\
    \hline
  \end{tabular}
    \caption{Comparison of loadings of principal components and explained variances for standard PCA and sparse versions. Here, the label ``Other'' denotes sparse PCA based on SDP, DDP, and SDDP. Each of these gives the optimal sparse solution.}
\label{tab:sparse_pca_small_comparisons}

  \end{center}
  \end{table}

Next, we consider a larger-scale example. We generate five random $100 \times 100$ covariance matrices of rank $4$. Table \ref{tab:sparse_pca_large_comparisons} presents the running times, number of non-zero entries (NNZ) in the top principal component\footnote{Note that the entries of the PCs were thresholded lightly in order to remove spurious non-zero elements arising from numerical inaccuracies in the solvers.}, optimal value of the program \eqref{eq:sparse_pca_sdp} (Opt.), and the explained variance (Expl.). The optimal values for DDP and SDDP upper bound the SDP solution as expected, and are quite close in value. The number of non-zero entries and explained variances are comparable across the different methods. The running times are between $1100$ and $2000$ times faster for DDP in comparison to SDP, and between $900$ and $1400$ times faster for SDDP. Hence the example illustrates that one can obtain a very large speedup with the approach presented here, with only a small sacrifice in quality of the approximate sparse PCs.

\begin{table}
\footnotesize
\tabcolsep=2pt
\begin{center}
  \begin{tabular}{ | c || c | c | c | c || c | c | c | c || c | c | c | c | }
    \hline 
    & \multicolumn{4}{|c||}{{\bf{DDP}}} & \multicolumn{4}{|c||}{{\bf{SDDP}}} & \multicolumn{4}{|c|}{{\bf{DSPCA}}} \\
    \hline
    No. & Time (s) & NNZ & Opt. & Expl.     & Time (s) & NNZ & Opt. & Expl.    & Time (s)   & NNZ & Opt. & Expl. \\ \hline
    1 & 0.89       & 4 & 35.5   & 0.067 \%  & 1.40     & 14  & 30.6 & 0.087 \% & 1296.9     & 5   & 30.5 & 0.086 \% \\ 
    2 & 1.18       & 4 & 36.6   & 0.071 \%  & 1.42     & 14  & 34.4 & 0.099 \% & 1847.3     & 6   & 33.6 & 0.092 \% \\ 
    3 & 1.46       & 4 & 49.5   & 0.079 \%  & 1.40     & 24  & 41.3 & 0.097 \% & 1633.0     & 16  & 40.2 & 0.096 \% \\ 
    4 & 1.14       & 4 & 44.1   & 0.072 \%  & 1.80     & 17  & 38.7 & 0.10 \%  & 1984.7     & 8   & 37.6 & 0.091 \% \\
    5 & 1.10       & 4 & 36.7   & 0.060 \%  & 1.53     & 34  & 33.0 & 0.068 \% & 2179.6     & 10  & 31.7 & 0.105 \% \\
    \hline
  \end{tabular}
    \caption{Comparison of running times, number of non-zero elements, optimal values, and explained variances for five large-scale sparse PCA examples using covariance matrices of size $100 \times 100$.}
\label{tab:sparse_pca_large_comparisons}

  \end{center}
  \end{table}

\subsection{Applications in control theory}
\label{sec:controls_apps}

As a final application of the methods presented in this paper, we consider two examples from control theory and robotics. The applications of SOS programming in control theory are numerous and include the computation of regions of attraction of polynomial systems \cite{PhD:Parrilo}, feedback control synthesis \cite{some_control_apps_sos}, robustness analysis \cite{Topcu10}, and computation of the joint spectral radius for uncertain linear systems \cite{Parrilo08}. A thorough treatment of the application of the (S)DSOS approach to control problems is beyond the scope of this paper but can be found in a different paper \cite{Majumdar14a}, which is joint work with Tedrake. Here we briefly highlight two examples from \cite{Majumdar14a} that we consider particularly representative.

Most of the applications in control theory mentioned above rely on SOS programming for checking \emph{Lyapunov inequalities} that certify stability of a nonlinear system. As discussed in Section \ref{sec:intro_applications}, one can compute (inner approximations of) the region of attraction (ROA) of a dynamical system $\dot{x} = f(x)$ by finding a Lyapunov function $V:\mathbb{R}^n\rightarrow\mathbb{R}$ that satisfies the following conditions:

\begin{equation}\label{eq:lyap_inequalities}
\begin{aligned}
&V(x)>0\quad \forall x\neq0,  \quad \quad \mbox{and}\\ &\dot{V}(x)=\langle\nabla V(x),f(x)\rangle<0\quad \forall x\in\{x|\ V(x)\leq\beta, x\neq 0\}.
\end{aligned}
\end{equation}

This guarantees that the set $\{x|\ V(x)\leq\beta\}$ is a subset of the ROA of the system, i.e. initial conditions that begin in this sublevel set converge to the origin. For polynomial dynamical systems, one can specify sufficient conditions for \eqref{eq:lyap_inequalities} using SOS programming \cite{PhD:Parrilo}. By replacing the SOS constraints with (S)DSOS constraints, we can compute inner approximations to the ROA more efficiently. While the approximations will be more conservative in general, the ability to tradeoff conservatism with computation time and scalability is an important one in control applications.

\subsubsection{Region of attraction for an inverted $N$-link pendulum}\label{subsub:ROA}

As an illustration, we consider the problem of computing ROAs for the underactuated $N$-link pendulum depicted in Figure \ref{fig:6link_cartoon}. This system has $2N$ states $x = [\theta_1,\dots,\theta_N,\dot{\theta}_1,\dots,\dot{\theta}_N]$ composed of the joint angles (with the vertical line) and their derivatives. There are $N-1$ control inputs (the joint closest to the base is not actuated). We take the unstable ``upright'' position of the system to be the origin of our state space and design a Linear Quadratic Regulator (LQR) controller in order to stabilize this equilibrium. A polynomial approximation of the dynamics of the closed loop system is obtained by a Taylor expansion of order $3$. Using Lyapunov functions of degree 2 results in the time derivative of the Lyapunov function being quartic and hence yields dsos/sdsos/sos constraints on a polynomial of degree 4 in $2N$ variables.

\begin{figure}
\minipage{0.2\textwidth}
\centering
\includegraphics[trim = 20mm 0mm 20mm 0mm, clip, width=.3\textwidth]{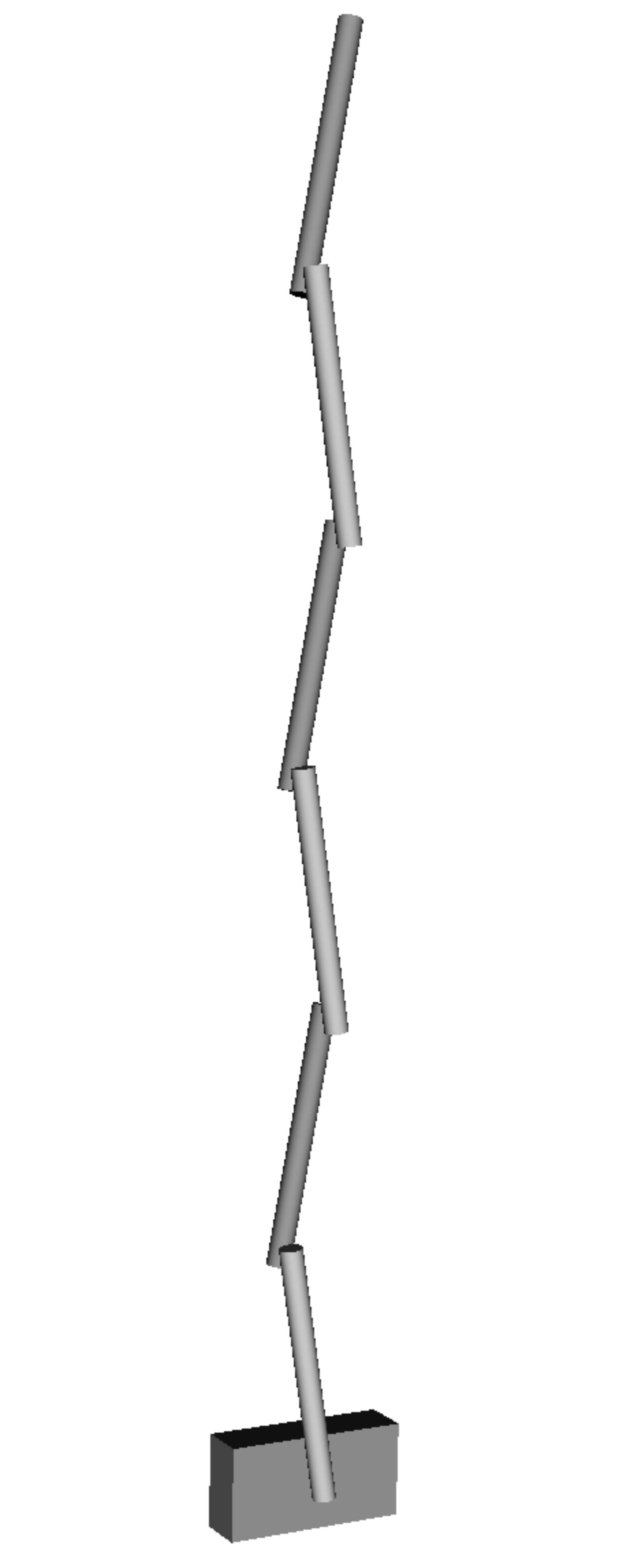}
\caption{An illustration of the N-link pendulum system (with N = 6).}
\label{fig:6link_cartoon}
\endminipage \hfill
\minipage{0.1\textwidth}
\endminipage
\minipage{0.75\textwidth}
\centering
\subfigure[$\theta_1$-$\dot{\theta}_1$ subspace.]{%
\includegraphics[trim = 15mm 70mm 20mm 70mm, clip, width=.38\textwidth]{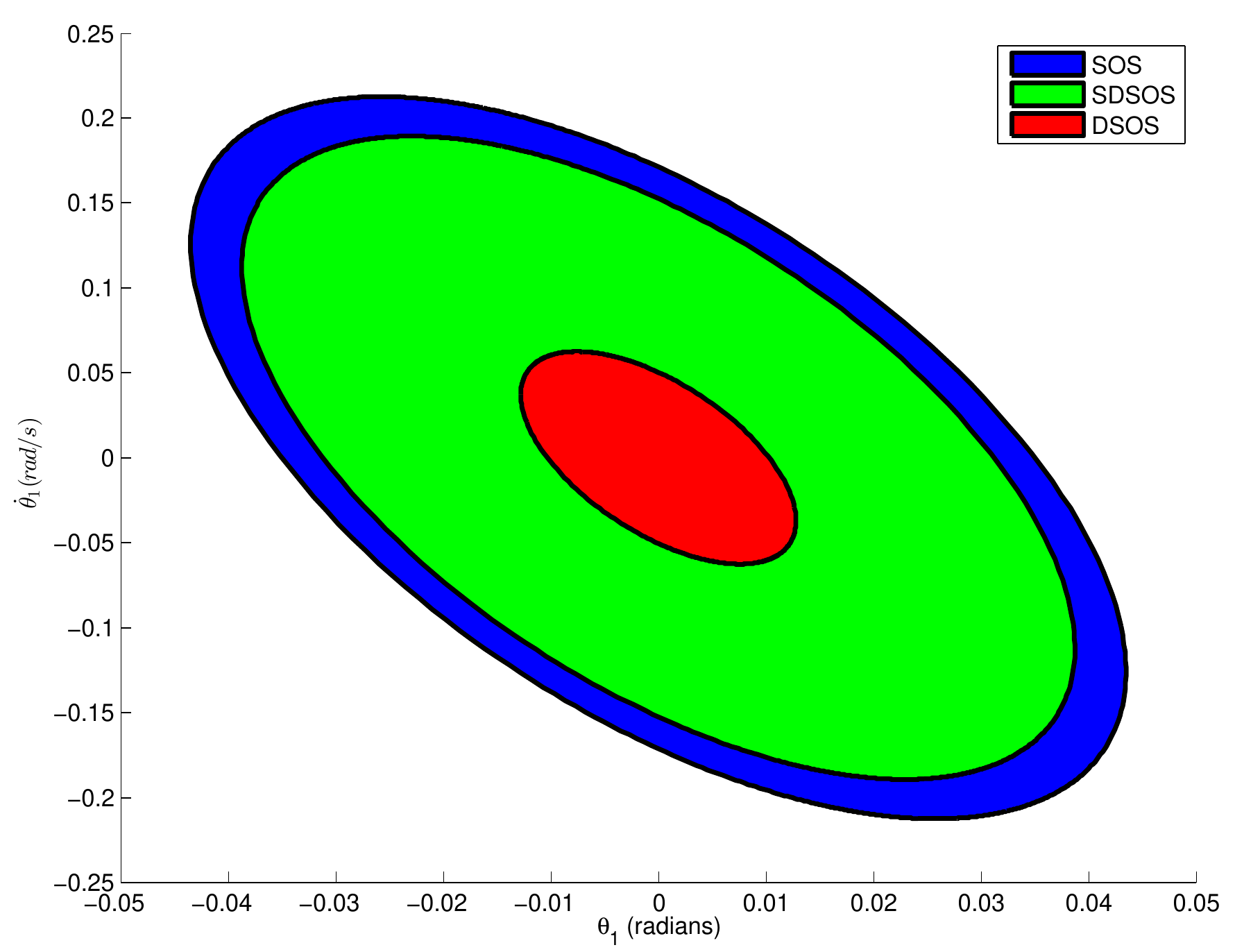}
\label{fig:6link_roa1}}
\subfigure[$\theta_6$-$\dot{\theta}_6$ subspace.]{%
\includegraphics[trim = 15mm 70mm 20mm 70mm, clip, width=.38\textwidth]{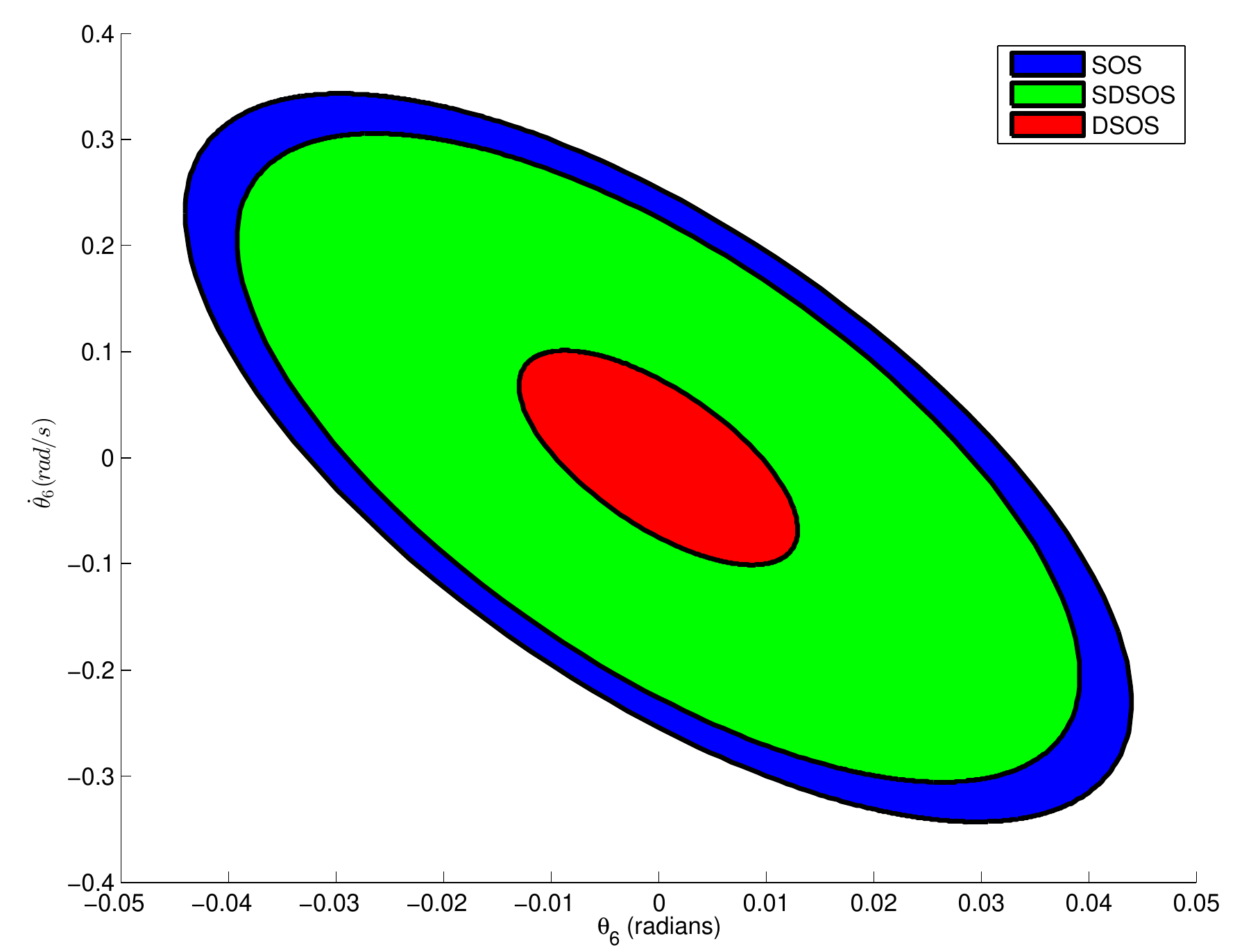}
\label{fig:6link_roa6}}

\caption{Figure reproduced from \cite{Majumdar14a} comparing projections of the ROAs computed for the 6-link pendulum system using DSOS, SDSOS and SOS programming.}
\label{fig:nlink_system}
\endminipage
\end{figure}

Figures \ref{fig:6link_roa1} and \ref{fig:6link_roa6} compare projections of the ROAs computed for the system with $N=6$ onto two $2-$dimensional subspaces of the state space. As the plots suggest, the ROA computed using SDSOS is only slightly more conservative than the ROA computed using SOS programming. In fact, comparing the volumes of the two ROAs, we find:

\begin{equation}
\frac{Vol_{\textrm{ROA-sdsos}}^{1/2N}}{Vol_{\textrm{ROA-sos}}^{1/2N}} = 0.79.
\end{equation}

Here, the volumes have been correctly rescaled in the standard manner by the dimension of the ambient space. The ROA computed using DSOS is much smaller (but still potentially useful in practice). 

Comparing running times of the different approaches on this example, we find that the LP corresponding to the DSOS program takes 9.67 seconds, and the SOCP corresponding to the SDSOS program takes 25.9 seconds. The running time of the SOS program is 1526.5 seconds using MOSEK and 23676.5 seconds using SeDuMi. Hence, in particular DSOS is approximately $2500$ times faster than SOS using SeDuMi and $150$ times faster than SOS using MOSEK, while SDSOS is $900$ times faster in comparison to SeDuMi and $60$ times faster than MOSEK for SOS. 

Of course the real benefit of our approach is that we can scale to problems where SOS programming ceases to run due to memory/computation constraints. Table \ref{tab:nlink_runtimes} illustrates this ability by comparing running times of the programs obtained using our approach with SOS programming for different values of $2N$ (number of states). As expected, for cases where the SOS programs do run, the DSOS and SDSOS programs are significantly faster. Further, the SOS programs obtained for $2N > 12$ are too large to run (due to memory constraints). In contrast, our approach allows us to handle almost twice as many states.

\begin{center}
\begin{table*}[ht]
{\small
	\tabcolsep=2pt
\hfill{}
 \begin{tabular}{ | l | c | c | c | c | c | c | c | c | c | c | }
    \hline 
    2N ($\#$ states) & 4 & 6 & 8 & 10 & 12 & 14 & 16 & 18 & 20 & 22 \\ \hline
    DSOS & $< 1$ & 0.44 & 2.04 & 3.08 & 9.67 & 25.1 & 74.2 & 200.5 & 492.0 & 823.2  \\ \hline
    SDSOS & $< 1$ & 0.72 & 6.72 & 7.78 & 25.9 & 92.4 & 189.0 & 424.74& 846.9 & 1275.6  \\ \hline
    SOS (SeDuMi) & $< 1$ & 3.97 & 156.9 & 1697.5 & 23676.5 & $\infty$ & $\infty$ & $\infty$ & $\infty$ & $\infty$ \\ \hline
    SOS (MOSEK) & $< 1$ & 0.84 & 16.2 & 149.1 & 1526.5 & $\infty$ & $\infty$ & $\infty$ & $\infty$ & $\infty$ \\ 
    \hline
\end{tabular}}
\hfill{}
\caption{Table reproduced from \cite{Majumdar14a} showing runtime comparisons (in seconds) for ROA computations on N-link system.}
\label{tab:nlink_runtimes}
\end{table*}
\vspace{-20pt}
\end{center}

\subsubsection{Control synthesis for  a humanoid robot}

\def\big{\includegraphics[trim = 100mm 150mm 200mm 10mm, clip, width=0.3\textwidth]{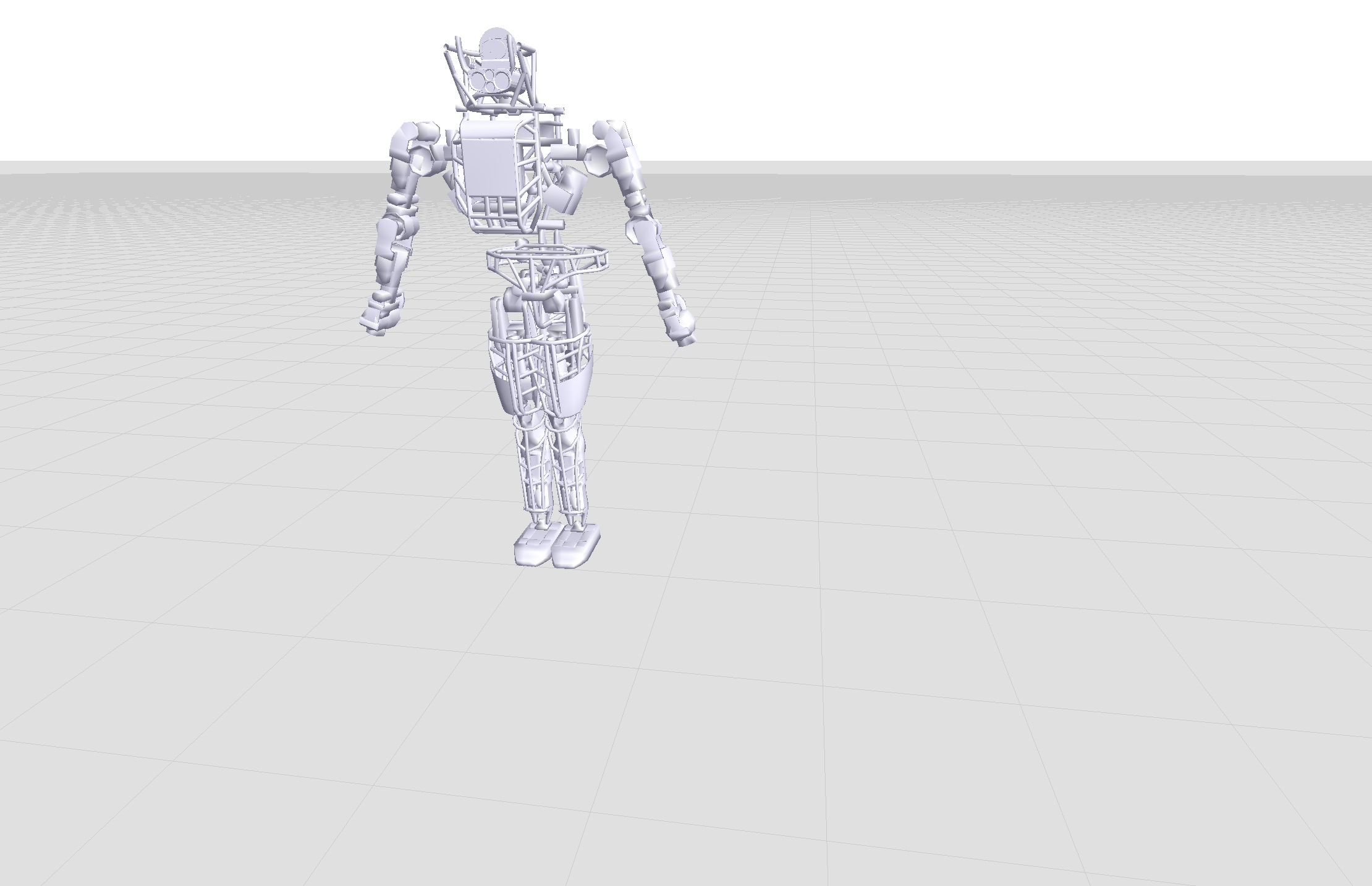}}
\def\little{\includegraphics[width=0.15\textwidth]{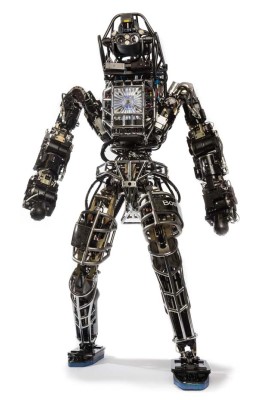}}
\def\stackalignment{l}
\begin{figure}
\centering
\bottominset{\little}{\big}{4pt}{160pt}
\caption{A visualization of the model of the ATLAS humanoid robot, along with the hardware platform (inset) on which the parameters of the model are based. (Picture of robot reproduced with permission from Boston Dynamics.)}
\label{fig:atlas_hardware}
\end{figure}

In our final example, we highlight a robotics application considered in joint work with Tedrake in \cite{Majumdar14a}. Control of humanoid robots is an important problem in robotics and presents a significant challenge due to the nonlinear dynamics of the system and high dimensionality of the state space. Here we show how the approach described in this paper can be used to design a balancing controller for a model of the ATLAS robot shown in Figure \ref{fig:atlas_hardware}. This robot was designed and built by Boston Dynamics Inc. and was used for the 2015 DARPA Robotics Challenge.

Our model of the robot is based on physical parameters of the hardware platform and has $30$ states and $14$ inputs. The task considered here is to balance the robot on its right toe. The balancing controller is constructed by searching for both a quadratic Lyapunov function and a linear feedback control law in order to maximize the size of the resulting region of attraction. The Lyapunov conditions are imposed using SDSOS programming. We Taylor expand the dynamics about the equilibrium to degree $3$ in order to obtain polynomial dynamics. The total computation time is approximately $22.5$ minutes. We note that SOS programming is unable to handle this system due to memory (RAM) constraints. 

Figure \ref{fig:atlas_init_conds} demonstrates the performance of the resulting controller from SDSOS programming by plotting initial configurations of the robot that are stabilized to the fixed point. As the plot illustrates, the controller is able to stabilize a very wide range of initial conditions. A video of simulations of the closed loop system started from different initial conditions is available online at \href{http://youtu.be/lmAT556Ar5c}{http://youtu.be/lmAT556Ar5c}.

\begin{figure}
\centering
\subfigure[Nominal pose (fixed point)]{%
\includegraphics[trim = 250mm 200mm 250mm 10mm, clip, width=.128\textwidth]{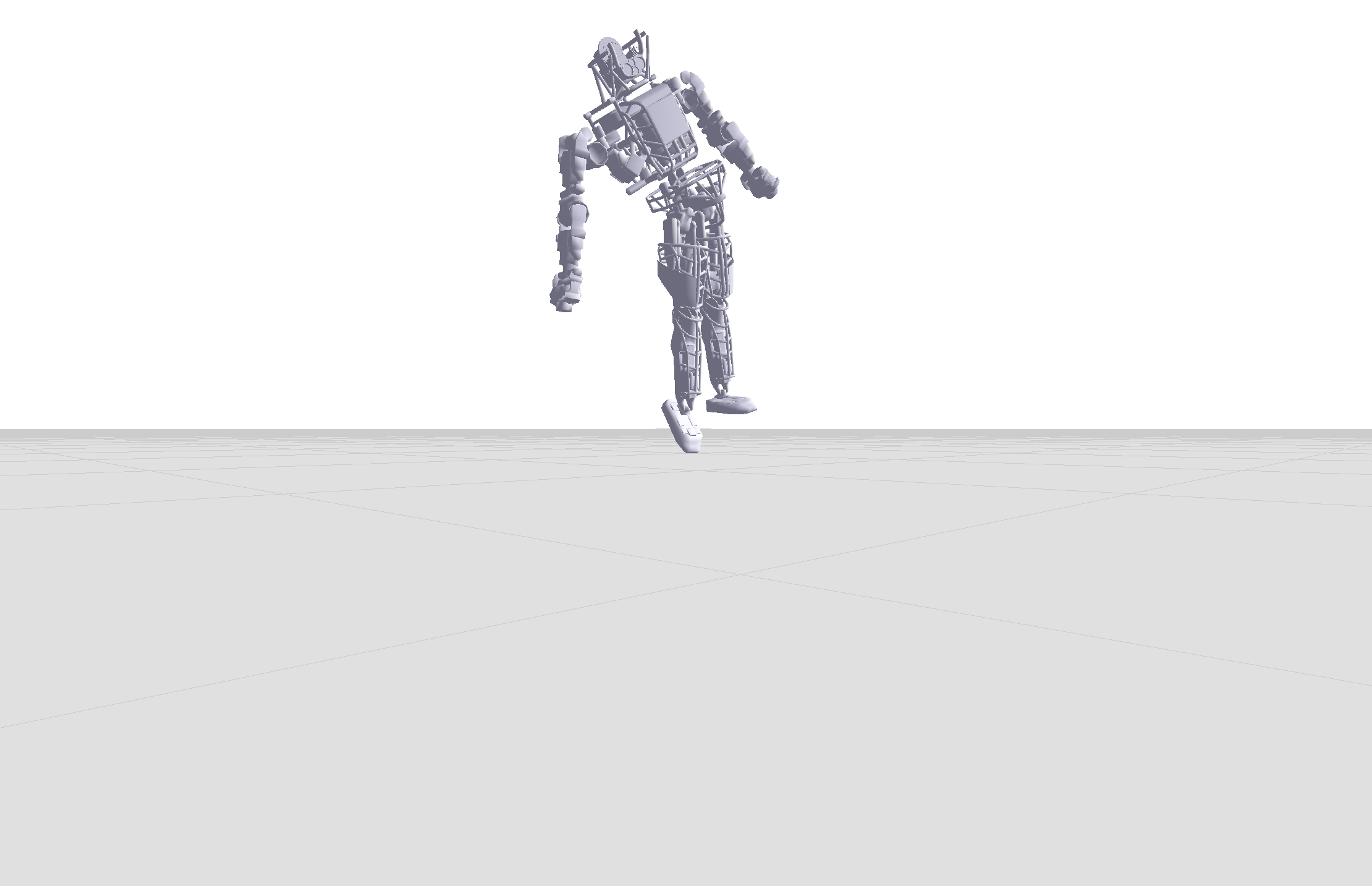}
\label{fig:atlas_pose_nominal}}
\hfill
\subfigure[Stabilized pose 1]{%
\includegraphics[trim = 250mm 200mm 250mm 10mm, clip, width=.128\textwidth]{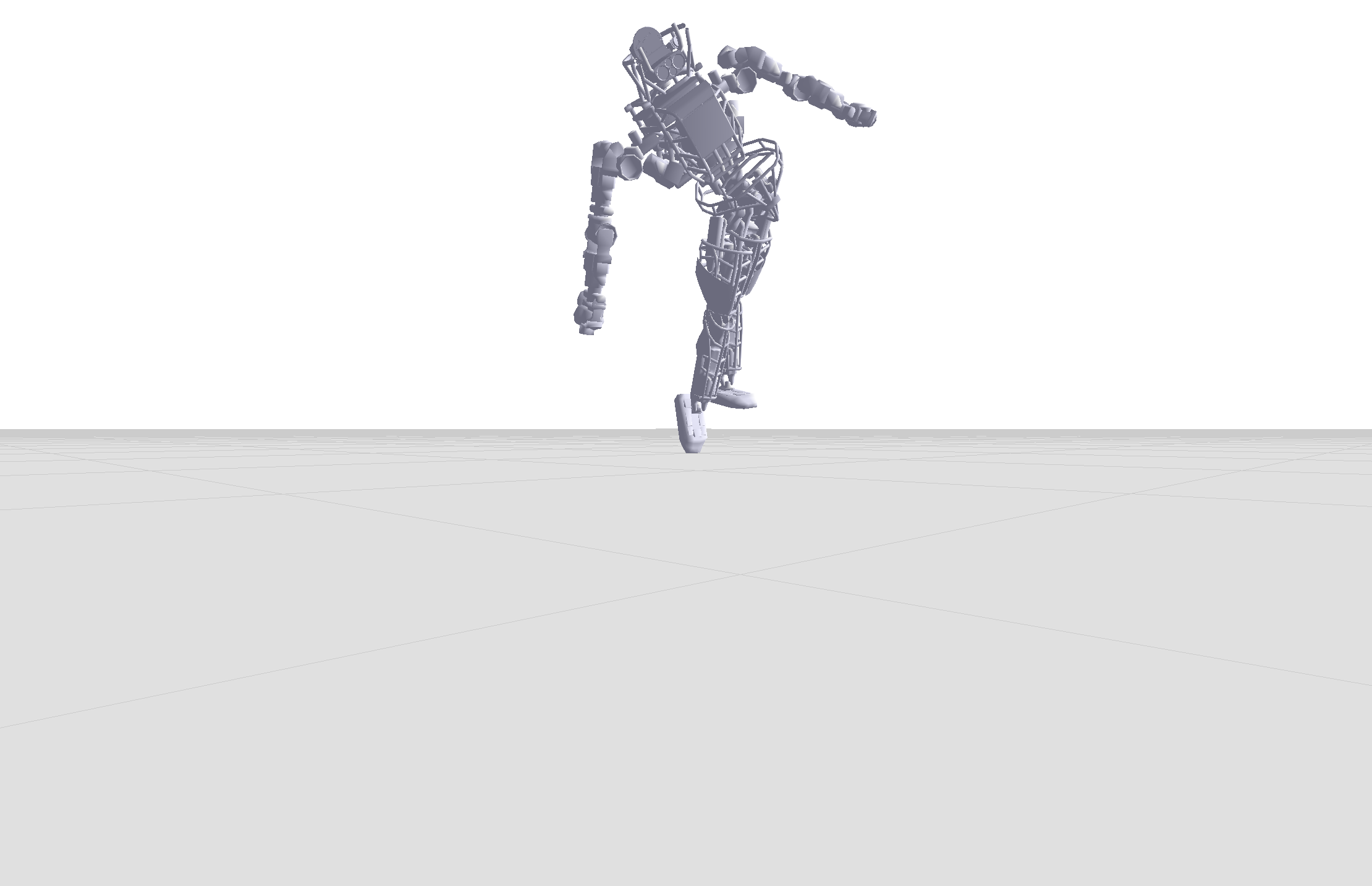}
\label{fig:atlas_pose1}}
\hfill
\subfigure[Stabilized pose 2]{%
\includegraphics[trim = 250mm 200mm 250mm 10mm, clip, width=.128\textwidth]{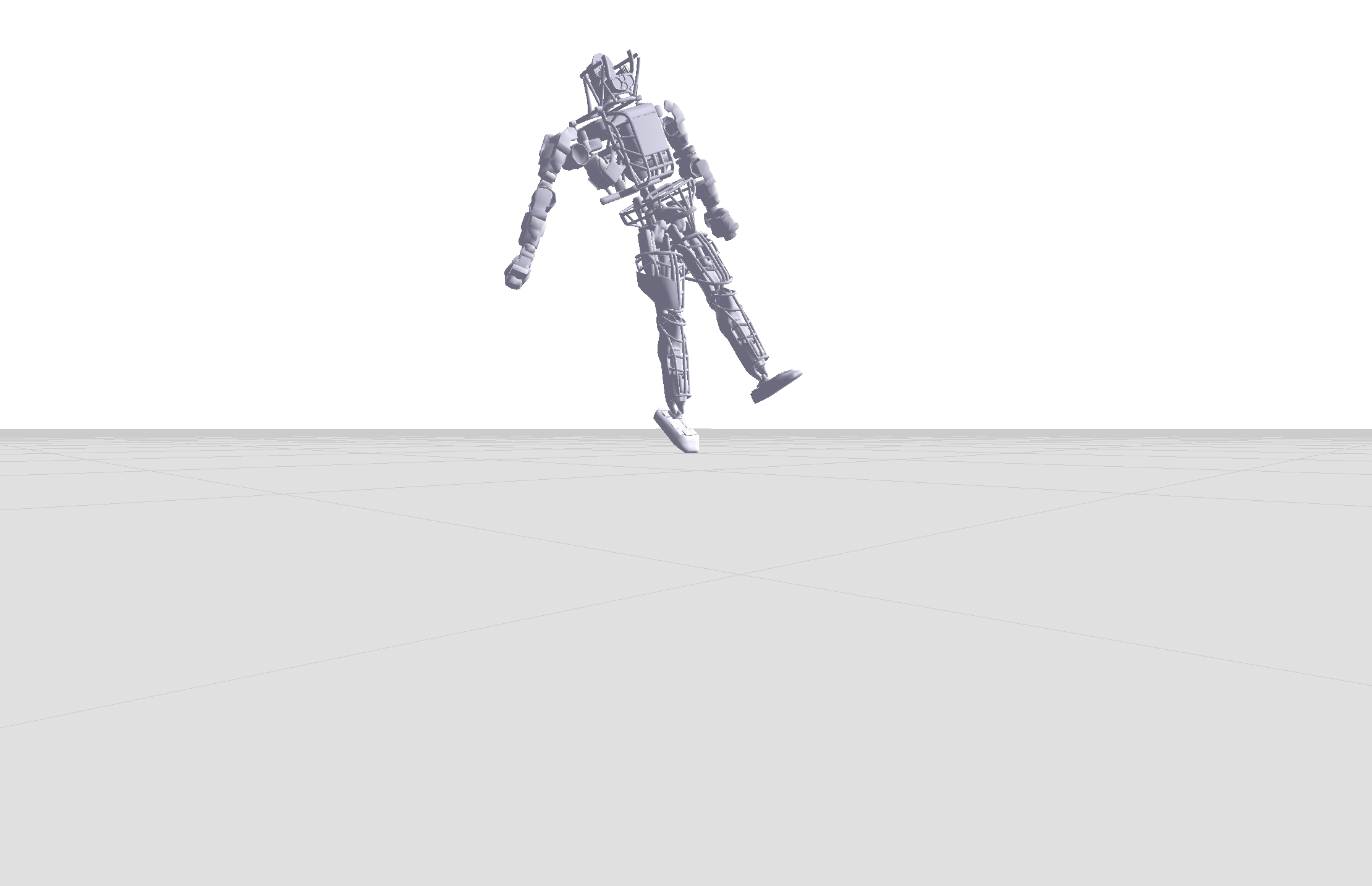}
\label{fig:atlas_pose2}}
\hfill
\subfigure[Stabilized pose 3]{%
\includegraphics[trim = 250mm 200mm 250mm 10mm, clip, width=.128\textwidth]{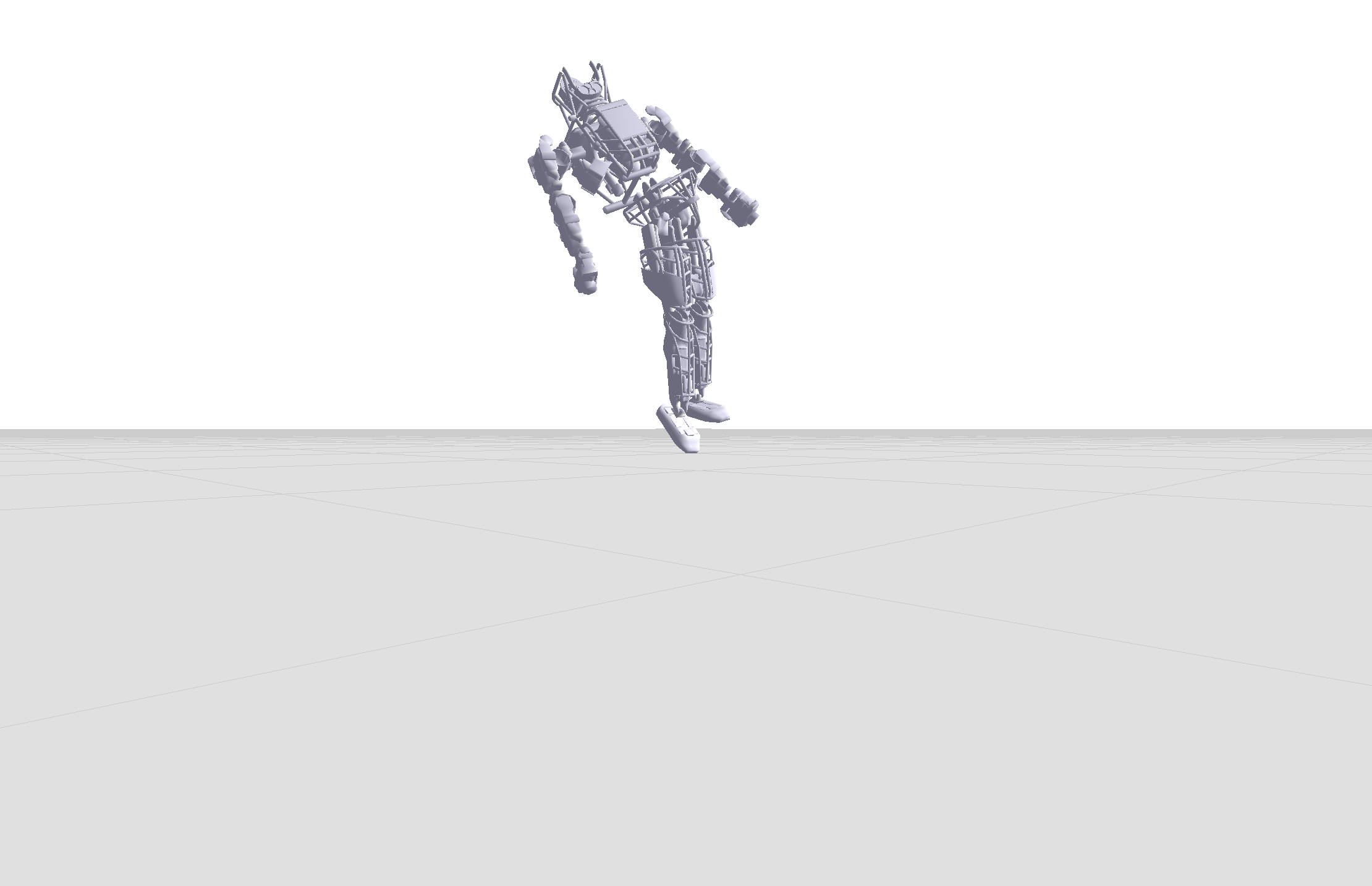}
\label{fig:atlas_pose3}}
\hfill
\subfigure[Stabilized pose 4]{%
\includegraphics[trim = 250mm 200mm 250mm 10mm, clip, width=.128\textwidth]{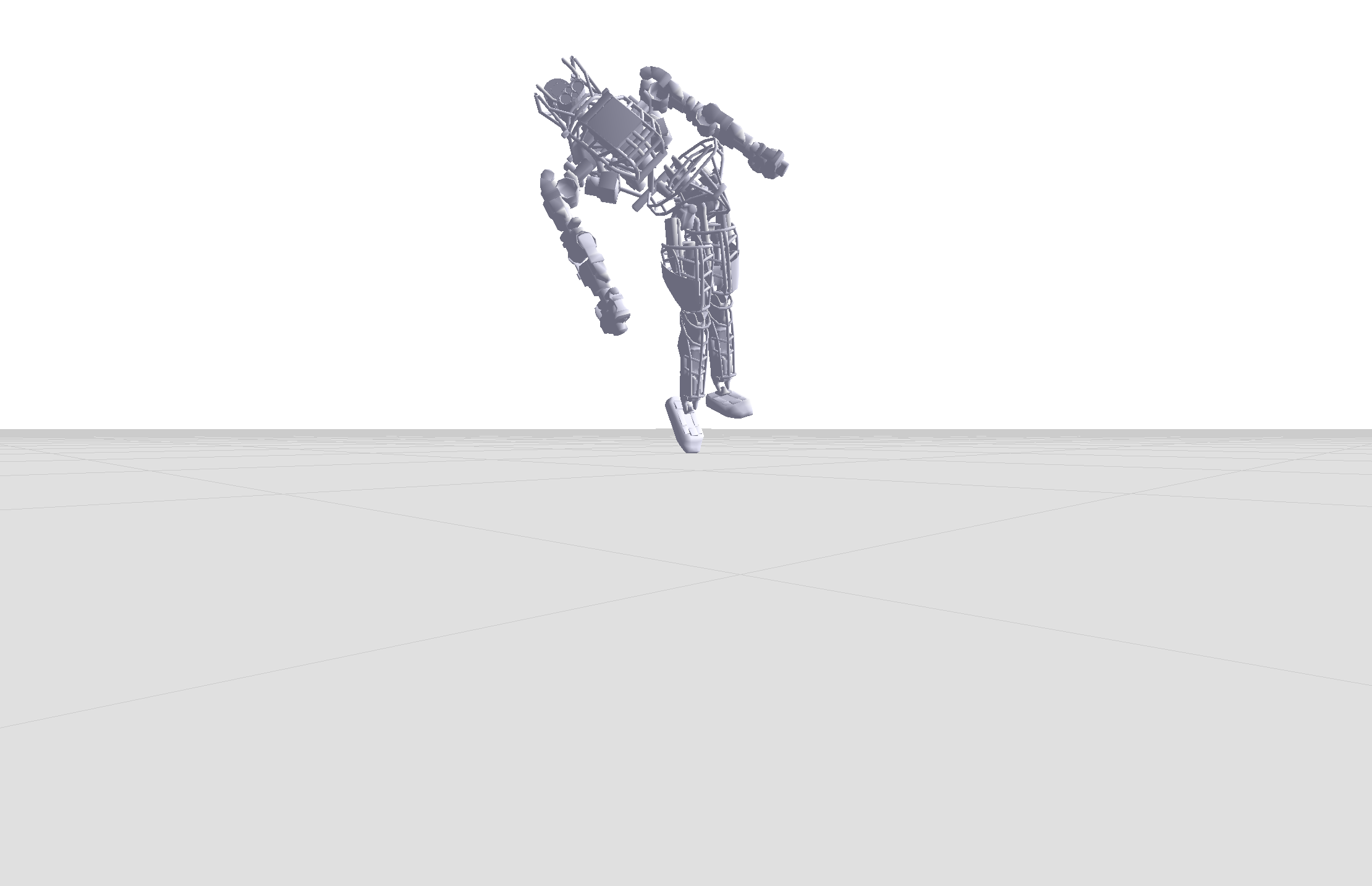}
\label{fig:atlas_pose4}}
\hfill
\subfigure[Stabilized pose 5]{%
\includegraphics[trim = 250mm 200mm 250mm 10mm, clip, width=.128\textwidth]{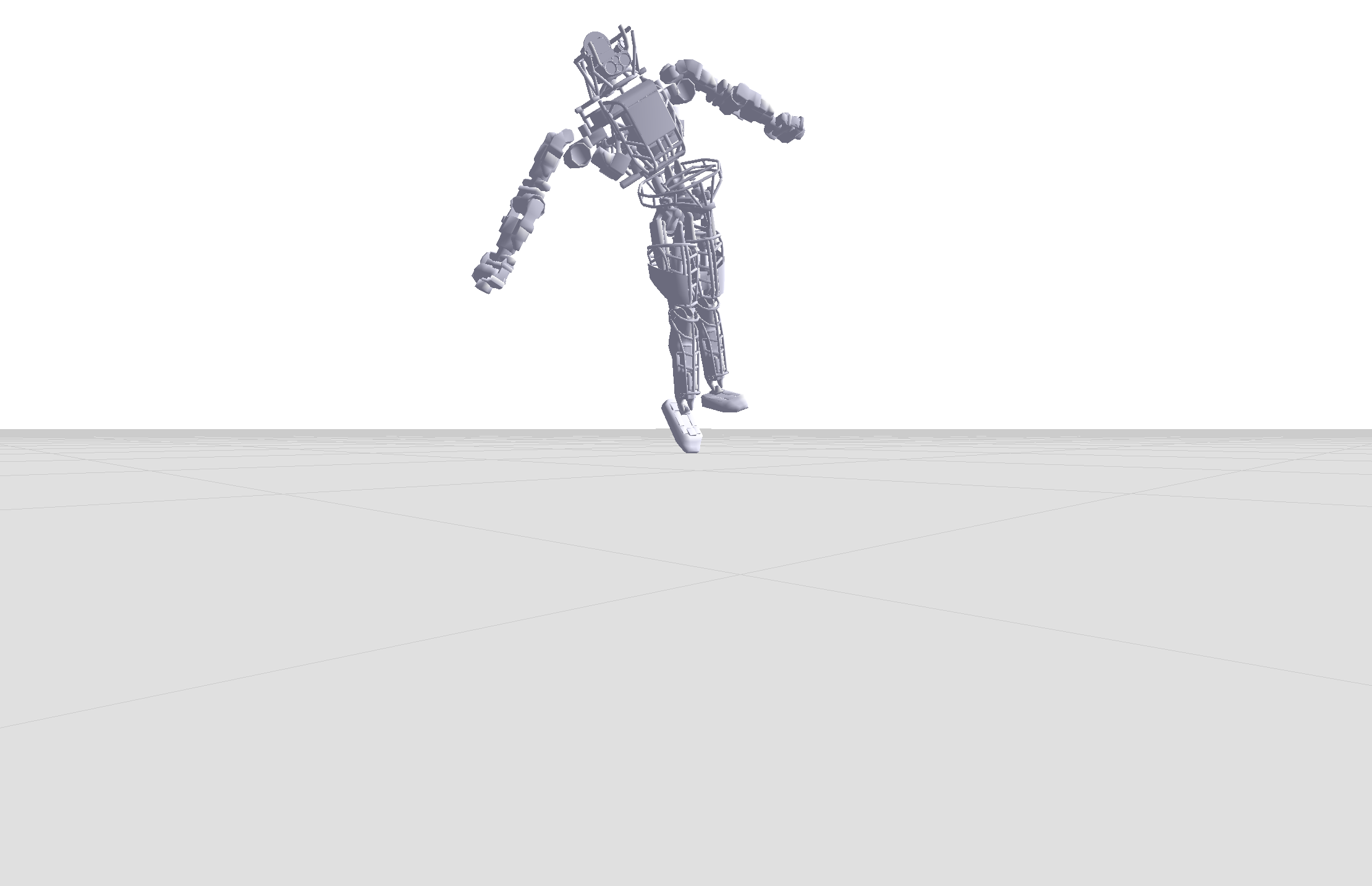}
\label{fig:atlas_pose5}}
\hfill
\caption{Figure reproduced from \cite{Majumdar14a} showing the nominal position of the robot, i.e., the fixed point being stabilized (subplot (a)), and configurations of the robot that are stabilized by the controller designed using SDSOS programing (subplots (b)-(f)). A video of simulations of the controller started from different initial conditions is available online at \href{http://youtu.be/lmAT556Ar5c}{http://youtu.be/lmAT556Ar5c}. }
\label{fig:atlas_init_conds}
\vspace{-15pt}
\end{figure}

\section{Improvements on DSOS and SDSOS optimization}\label{sec:improvements}

While DSOS and SDSOS techniques result in significant gains in terms of solving times and scalability, they inevitably lead to some loss in solution accuracy when compared to the SOS approach. In this section, we briefly outline three possible strategies to mitigate this loss. The first two (Sections~\ref{subsec:cholesky} and \ref{subsec:column.generation}) generate sequences of linear and second order cone programs, while the third (Section~\ref{subsec:factor.width}) works with ``small'' (fixed-size) semidefinite programs. The reader may recall that the r-DSOS and r-SDSOS hierarchies of Section~\ref{subsec:rdsos.rsdsos} could also be used to improve on DSOS and SDSOS techniques. Like the first two methods that we present here, they did so while staying in the realm of LP and SOCP. They differ from these methods on two crucial points however: (i) they are more expensive to implement, and (ii) they approximate the cone of SOS polynomials ``blindly'' irrespective of a particular objective function. This is in contrast to the methods that we present in Section \ref{subsec:cholesky} and \ref{subsec:column.generation}, which approximate the cone in the direction of a specific linear objective function.

For brevity of exposition, we explain how all three strategies can be applied to approximate a generic semidefinite program:
\begin{flalign} 
\label{eq:sdp}
			\underset{X\in\mathcal{S}_n}{\text{minimize}} \hspace*{1cm} & \mbox{Tr}(CX)  \\
			\text{s.t.} \hspace*{1cm} & \mbox{Tr}(A_iX)=b_i,\quad  i=1,\ldots,m, \nonumber \\
			& X\succeq 0. \nonumber
\end{flalign}

A treatment tailored to the case of sum of squares programs can be found in the references we provide.


\subsection{Iterative change of basis}\label{subsec:cholesky} In~\cite{basis_pursuit_Ahmadi_Hall}, Ahmadi and Hall build on the notions of diagonal and scaled diagonal dominance to provide a sequence of improving inner approximations to the cone $P_n$ of psd matrices in the direction of the objective function of an SDP at hand. The idea is simple: define a family of cones\footnote{One can think of $DD(U)$ as the set of matrices that are dd after a change of coordinates via the matrix $U$.} $$DD(U)\mathrel{\mathop{:}}=\{M \in \mathcal{S}_n~|~ M=U^TQU \text{ for some dd matrix } Q \},$$parametrized by an $n \times n$ matrix $U$. Optimizing over the set $DD(U)$ is an LP since $U$ is fixed, and the defining constraints are linear in the coefficients of the two unknown matrices $M$ and $Q$. Furthermore, the matrices in $DD(U)$ are all psd; i.e., $\forall U,$ $DD(U) \subseteq P_n$.

The proposal in~\cite{basis_pursuit_Ahmadi_Hall} is to solve a sequence of LPs, indexed by $k$, by replacing the condition $X \succeq 0$ by $X \in DD(U_k)$:

\begin{flalign} 
\label{eq:LPChol}
			DSOS_k\mathrel{\mathop{:}}=\underset{X\in\mathcal{S}_n}{\text{min}} \hspace*{1cm} & \mbox{Tr}(CX)  \\
			\text{s.t.} \hspace*{1cm} & \mbox{Tr}(A_iX)=b_i,\quad  i=1,\ldots,m, \nonumber \\
			& X\in DD(U_k). \nonumber
\end{flalign}


The sequence of matrices $\{U_k\}$ is defined as follows
\begin{equation}\label{eq:defUk}
\begin{aligned}
U_0&=I\\
U_{k+1}&=\text{chol}(X_k),
\end{aligned}
\end{equation}
where $X_k$ is an optimal solution to the LP in (\ref{eq:LPChol}).

Note that the first LP in the sequence optimizes over the set of diagonally dominant matrices. By defining $U_{k+1}$ as a Cholesky factor of $X_k$, improvement of the optimal value is guaranteed in each iteration. Indeed, as $X_k=U_{k+1}^T I U_{k+1}$, and the identity matrix $I$ is diagonally dominant, we see that $X_{k} \in DD(U_{k+1})$ and hence is feasible for iteration $k+1$. This implies that the optimal value at iteration $k+1$ is at least as good as the optimal value at the previous iteration; i.e., $DSOS_{k+1}\leq DSOS_k$ (in fact, the inequality is strict under mild assumptions; see~\cite{basis_pursuit_Ahmadi_Hall}).

In an analogous fashion, one can construct a sequence of SOCPs that inner approximate $P_n$ with increasing quality. This time, we define a family of cones $${ SDD(U)\mathrel{\mathop{:}}=\{M \in \mathcal{S}_n ~|~ M=U^TQU, \text{ for some sdd matrix } Q\},}$$ parameterized again by an $n\times n$ matrix $U$. For any $U$, optimizing over the set $SDD(U)$ is an SOCP and we have $SDD(U)\subseteq P_n$. This leads us to the following iterative SOCP sequence:

\begin{flalign} 
\label{eq:SOCPchol}
			SDSOS_k\mathrel{\mathop{:}}=\underset{X\in\mathcal{S}_n}{\text{min}} \hspace*{1cm} & \mbox{Tr}(CX)  \\
			\text{s.t.} \hspace*{1cm} & \mbox{Tr}(A_iX)=b_i,\quad  i=1,\ldots,m, \nonumber \\
			& X\in SDD(U_k). \nonumber
\end{flalign}


Assuming existence of an optimal solution $X_k$ at each iteration, we can define the sequence $\{U_k\}$ iteratively in the same way as was done in (\ref{eq:defUk}). Using similar reasoning, we have $SDSOS_{k+1}\leq SDSOS_k$. In practice, the sequence of upper bounds $\{SDSOS_k\}$ approaches faster to the SDP optimal value than the sequence of the LP upper bounds $\{DSOS_k\}$.
\begin{figure}[h!]
	\begin{center}
		\mbox{
			\subfigure[LP inner approximations]
			{\label{subfig:DDCones}\scalebox{0.32}{\includegraphics{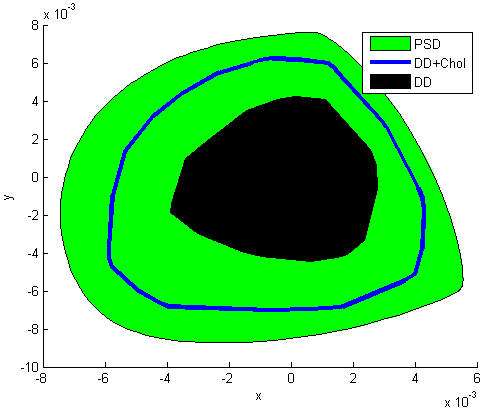}}}}
		\mbox{
			\subfigure[SOCP inner approximations]
			{\label{subfig:SDDCones}\scalebox{0.32}{\includegraphics{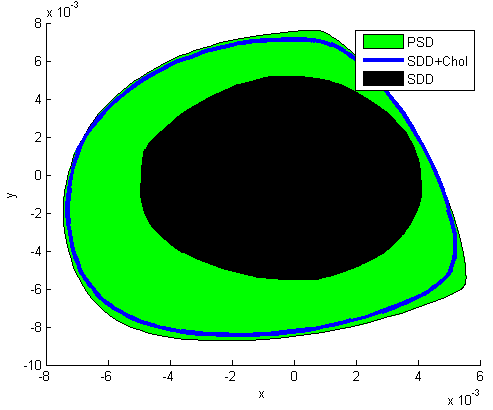}}}
		}
		
		\caption{Figure reproduced from~\cite{basis_pursuit_Ahmadi_Hall} showing improvement (in all directions) after one iteration of the change of basis algorithm.}
		\label{fig:DDSDDCones}
	\end{center}
\end{figure}
Figure~\ref{fig:DDSDDCones} shows the improvement (in every direction) obtained just by a single iteration of this approach. The outer set in green in both subfigures is the feasible set of a randomly generated semidefinite program. The sets in black are the DD (left) and the SDD (right) inner approximations. What is shown in dark blue in both cases is the boundary of the improved inner approximation after one iteration. Note that the SOCP in particular fills up almost the entire spectrahedron in a single iteration.

We refer the interested reader to~\cite{basis_pursuit_Ahmadi_Hall} for more details, in particular for an explanation of how the same techniques can be used (via duality) to \emph{outer approximate} feasible sets of semidefinite programs.

\subsection{Column generation}\label{subsec:column.generation}
In~\cite{col_gen_Ahmadi_Dash_Hall}, Ahmadi, Dash, and Hall design another iterative method for inner approximating the set of psd matrices using linear and second order cone programming. Their approach combines DSOS/SDSOS techniques with ideas from the theory of column generation in large-scale linear and integer programming. The high-level idea is to approximate the SDP in (\ref{eq:sdp}) by a sequence of LPs (parameterized by $t$):

\begin{flalign} 
\label{eq:LP.colgen}
			\underset{X\in\mathcal{S}_n, \alpha_i}{\text{minimize}} \hspace*{1cm} & \mbox{Tr}(CX)  \\
			\text{s.t.} \hspace*{1cm} & \mbox{Tr}(A_iX)=b_i,\quad  i=1,\ldots,m, \nonumber \\
			& X=\sum_{i=1}^t \alpha_i B_i, \nonumber \\
			\ & \alpha_i\geq 0, \quad i=1,\ldots, t,
\end{flalign}
where $B_1,\ldots,B_t$ are fixed psd matrices. These matrices are initialized to be the extreme rays of $DD_n$ (recall Lemma~\ref{lem:dd.corners}), i.e., all rank one matrices $v_iv_i^T$, where the vector $v_i$ has at most two nonzero components, each equal to $\pm 1$. Once this initial LP is solved, then one adds one (or sometimes several) new psd matrices $B_j$ to problem (\ref{eq:LP.colgen}) and resolves. This process then continues. In each step, the new matrices $B_j$ are picked carefully to bring the optimal value of the LP closer to that of the SDP. Usually, the construction of $B_j$ involves solving a ``\emph{pricing subproblem}'' (in terminology of the column generation literature), which adds appropriate cutting planes to the dual of (\ref{eq:LP.colgen}); see~\cite{col_gen_Ahmadi_Dash_Hall} for more details.

The SOCP analogue of this process is similar. The SDP in (\ref{eq:sdp}) is inner approximated by a sequence of SOCPs (parameterized by $t$):
\begin{flalign} 
\label{eq:SOCP.colgen}
			\underset{X\in\mathcal{S}_n, \Lambda_i\in\mathcal{S}_2}{\text{minimize}} \hspace*{1cm} & \mbox{Tr}(CX)  \\
			\text{s.t.} \hspace*{1cm} & \mbox{Tr}(A_iX)=b_i,\quad  i=1,\ldots,m, \nonumber \\
			& X=\sum_{i=1}^t V_i\Lambda_i V_i^T, \nonumber  \\
						\ & \Lambda_i\succeq 0, \quad i=1,\ldots, t, \nonumber 
\end{flalign}
where $V_1,\ldots,V_t$ are fixed $n \times 2$ matrices. They are initialized as the set of matrices that have zeros everywhere, except for a 1 in the first column in position $j$ and a 1 in the second column in position $k\neq j$. This gives exactly $SDD_n$ (via Lemma~\ref{lem:sdd=sum.2x2}). In subsequent steps, one (or sometimes several) appropriately-chosen matrices $V_i$ are added to problem (\ref{eq:SOCP.colgen}). These matrices are obtained by solving pricing subproblems and help bring the optimal value of the SOCP closer to that of the SDP in each iteration.

\begin{figure}[h]
	\begin{center}
		\mbox{
			\subfigure[LP iterations.]
			{\label{subfig:dsos.iters}\scalebox{0.32}{\includegraphics{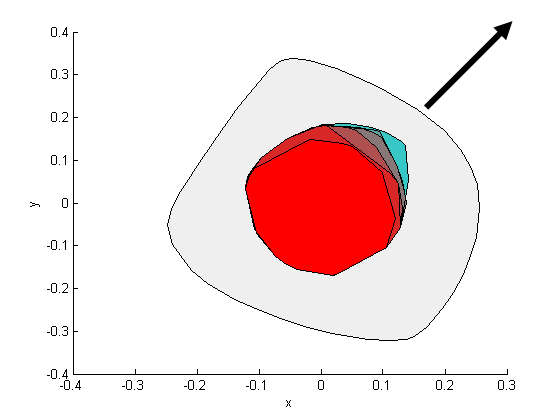}}}}
		\mbox{
			\subfigure[SOCP iterations.]
			{\label{subfig:sdsos.iters}\scalebox{0.32}{\includegraphics{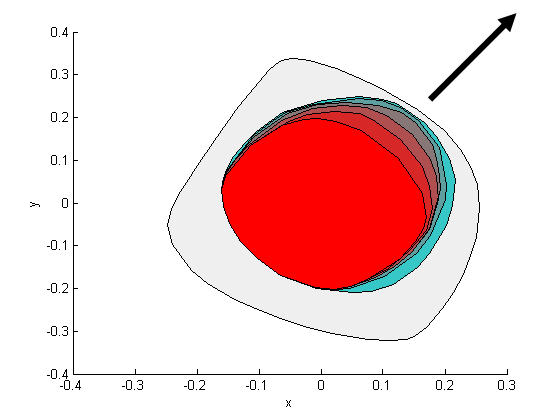}}}
		}
		
		\caption{Figure reproduced from~\cite{col_gen_Ahmadi_Dash_Hall} showing the successive improvement on the dd (left) and sdd (right) inner approximation of a spectrahedron via five iterations of the column generation method.}
		\label{fig:colgen}
	\end{center}
	\vspace{-10pt}
\end{figure}

Figure~\ref{fig:colgen} shows the improvement obtained by five iterations of this process on a randomly generated SDP, where the objective is to maximize a linear function in the northeast direction.
5
\subsection{Factor-width $k$ matrices with $k>2$}\label{subsec:factor.width}

A relevant generalization of sdd matrices is through the notion of \emph{factor-width}, defined by Boman et al.~\cite{sdd_factorwidth}. The factor-width of a symmetric psd matrix $Q$ is the smallest integer $k$ for which $Q$ can be written as $Q=VV^T$, where each column of $V$ contains at most $k$ nonzeros. Equivalently, the factor-width of $Q$ is the smallest $k$ for which $Q$ can be written as a sum of psd matrices that are nonzero only on a single $k\times k$ principal submatrix.

If we denote the cone of $n\times n$ symmetric matrices of factor-width $k$ by $FW^k_n$, we have that $FW^k_n\subseteq P_n$ for all $k\in\{1,\ldots,n\}$, $FW^2_n=SDD_n$ (cf. Lemma~\ref{lem:sdd=sum.2x2}), and $FW^n_n=P_n$. For values of $k\in\{2,\ldots,n-1\}$, one gets an increasingly more accurate inner approximation to the set of psd matrices, all outperforming the approximation given by sdd matrices. We did not consider these cones in this paper as it is already known that for $k=3$, a representation based on second order cone programming is not possible~\cite{fawzi_socp}. Nevertheless, working with $FW^3_n$ can be useful as it involves semidefinite constraints that are small ($3\times 3$) and hence efficiently solvable. Unfortunately though, optimizing over $FW^3_n$ requires $O(n^3)$ such semidefinite constraints which in many applications can be prohibitive. We believe that a promising future direction would be to use the column generation framework on Section~\ref{subsec:column.generation} to work with a subset of the extreme rays of $FW^k_n$ (for small $k$) and generate additional ones on the fly based on problem structure.

Some initial experiments with $FW^k_n$ have been performed by Ding and Lim~\cite{ding_lim}. The authors also develop self-concordant barrier functions for these cones and show that any second order cone program can be written as an optimization problem over $FW^2_n$, i.e., a scaled diagonally dominant program (recall the definition from Section~\ref{subsec:dsos.sdsos.cones}). In related work, Permenter and Parrilo~\cite{permenter_parrilo_facial_reduction} consider optimization problems over $FW^k_n$ (and their dual cones) for facial reduction in semidefinite programming. We believe the notion of factor-width is likely to receive more attention in upcoming years in the theory of matrix optimization.

\section{Conclusions}
\label{sec:conclusions}


We have proposed more scalable alternatives to SOS programming by introducing inner approximations to the sum of squares cone in the form of the cones of dsos and sdsos polynomials. These cones can be optimized over using LP and SOCP respectively \aaa{and afford us considerable gains in terms of scalability by trading off solution quality}. Our numerical examples from a diverse range of applications including polynomial optimization, combinatorial optimization, statistics and machine learning, derivative pricing, control theory, and robotics demonstrate that with reasonable tradeoffs in solution quality, we can handle problem sizes that are significantly beyond the current capabilities of SOS programming (at least without exploiting structure or resorting to specialized solvers). In particular, we have shown that our approach is able to produce useful lower bounds on dense polynomial optimization problems with up to $70$ polynomial variables (with degree-4 polynomials), or design stabilizing controllers for realistic robotic control systems with as many as 30 state variables. \aaa{The Supplementary Material of this paper provides a brief but complete tutorial on the toolbox that was used to generate the numerical results in the paper.} In addition, since the extended abstract of this work first appeared, our techniques have been applied by other authors to areas such as \aaa{conic optimization~\cite{Joao_sdd}}, distance geometry~\cite{dias2016diagonally}, portfolio selection in finance~\cite{bai2015alternating}, and power engineering~\cite{acopf_dsos,Leo_student_thesis}.


On the theoretical front, we have shown that the (S)DSOS approach shares some of the theoretical asymptotic guarantees usually associated with SOS programming (such as certificates of nonnegativity of forms, arbitrarily tight approximation of copositive programs, and a converging hierarchy of lower bounds for polynomial optimization problems). Finally, in the last section of this paper, we reviewed recent approaches that were developed with the idea of bridging the gap between the (S)DSOS approach and the SOS one. One can obtain improved \aaa{approximations} using these methods at the cost 
of additional computation time. \aaa{This raises as a natural research direction the possibility of \emph{quantifying} these tradeoffs, either for one-shot or for adaptive approximation schemes.}


To conclude, we would like to emphasize that a key benefit of our approach is that it can be applied in \emph{any} \aaa{application where SOS programming is used.} 
Our hope is that the (S)DSOS approach may open up application areas that have previously been beyond reach due to the limitations in scalability of SOS programming.  While it is true that improvements to SDP solver technology will always push the boundary of problems that are within the reach of SOS programming, we believe that concurrent improvements in LP and SOCP solver technology will always leave a range of additional problems that can be handled by relaxations such as ours. For example, we foresee exciting possibilities for real-time applications~\cite{ahmadi2016some},~\cite{boyd_real_time_convex_signal_processing} using technology for real-time linear and second order cone programming that is already coming to fruition~\cite{domahidi2013ecos},~\cite{cvxgen}.



{\footnotesize
\section*{Acknowledgments}
We would like to thank Pablo Parrilo for acquainting us with scaled diagonally dominant matrices, and Georgina Hall for many corrections and simplifications on the first draft of this work. Our gratitude extends to Russ Tedrake and the members of the Robot Locomotion Group at MIT for several insightful discussions, particularly around control applications. We thank Frank Permenter and Mark Tobenkin for their help with the numerical implementation of DSOS/SDSOS programs in SPOT, and Aida Khajavirad for running the experiments of Section~\ref{subsec:experiments.pop} with BARON. We thank the Editorial Board of SIAGA and the anonymous referees for their constructive feedback and insightful comments. Finally, this work has benefited from questions and comments from several colleagues, among whom we gladly acknowledge Greg Blekherman, Stephen Boyd, Sanjeeb Dash, Etienne de Klerk, Jesus de Loera, Lijun Ding, Jo\~ao Gouveia, Didier Henrion, Jean Bernard Lasserre, Leo Liberti, Lek-Heng Lim, Bruce Reznick, James Saunderson, and Bernd Sturmfels.}

\newpage

\section{Supplementary Material: Software Toolbox}
\label{sec:spot}

A complete implementation of the code used to generate the numerical results presented in this paper was written using the Systems Polynomial Optimization Toolbox (SPOT) \cite{Megretski_spot} and is freely available online\footnote{Link to the software:

\href{https://github.com/anirudhamajumdar/spotless/tree/spotless_isos}{https://github.com/anirudhamajumdar/spotless/tree/spotless\_isos}}. The toolbox features efficient polynomial algebra and allows us to setup the large-scale LPs and SOCPs arising from our examples. Here we provide a brief introduction to the software. This is not meant as a comprehensive tutorial on the SPOT toolbox (for this one may refer to SPOT's documentation\footnote{Link to SPOT's documentation:

\href{https://github.com/spot-toolbox/spotless/blob/master/doc/manual.pdf}{https://github.com/spot-toolbox/spotless/blob/master/doc/manual.pdf}}). Rather, the goal here is to provide an introduction sufficient for setting up and solving the DSOS and SDSOS programs in this paper. The code relevant for this purpose is included in a branch of SPOT that we have named \texttt{iSOS} (for ``inside sum of squares'').

\subsection{Installing SPOT}

Download the software package from Github available here:

\quad

\noindent \href{https://github.com/anirudhamajumdar/spotless/tree/spotless_isos}{https://github.com/anirudhamajumdar/spotless/tree/spotless\_isos}

\quad

\noindent Next, start MATLAB and run the \texttt{spot\_install.m} script. This script will setup the MATLAB path and compile a few mex functions.
The user may wish to save the new MATLAB path for future use.

\subsection{Variables and polynomials}

Polynomials are defined and manipulated by the \\ \texttt{@msspoly} class of \texttt{SPOT}. In order to define a new variable, one can use the \texttt{msspoly.m} function:

$$ >> \texttt{x = msspoly(`x')}$$
which creates the polynomial $p(x) = x$. The argument to this function is the name of the created variable and is restricted to four characters chosen from the alphabet (lower case and upper case). A MATLAB vector of variables can be created by passing a second argument to the function:

$$ >> \texttt{x = msspoly(`x',n)}.$$

This will create a $n \times 1$ vector of variables that can be accessed using standard array indexing. Multivariate polynomials can then be constructed as follows:

\begin{flalign}
& >> \texttt{x = msspoly(`x',3);} \nonumber \\
& >> \texttt{p = 2*x(1)\string^2 - 5*x(2)\string^2 + x(3)\string^2} \nonumber
\end{flalign}

Variables can be manipulated and operated on using a variety of functions. These include standard arithmetic operations (e.g. addition, multiplication, dot product) and operations for manipulating vectors (e.g. concatenating, reshaping, transposing). Other useful functions include:

\quad

\texttt{deg.m}: Returns the total degree of a polynomial. If a second argument in the form of a \texttt{msspoly}
is provided, the degree with respect to these variables is returned.

\quad

\texttt{diff.m}: Differentiates a polynomial (first argument) with respect to a set of \texttt{msspoly} variables (second argument). The result is the matrix of partial derivatives.

\quad

\texttt{subs.m}: Substitutes the third argument in place of the second argument wherever it appears in the first argument.

\subsection{Programs}

Programs involving DSOS/SDSOS/SOS constraints are constructed using the \texttt{spotsosprog} class. In this section, we demonstrate the workings of \texttt{spotsosprog} with the help of an example. In particular, we take the problem of minimizing a form on the unit sphere considered in this paper. The following example is also available in \texttt{doc/examples/sdsos\string_example.m}.

\begin{flalign*}
& \texttt{\% Construct polynomial which is to be minimized} \\ 
& \texttt{x = msspoly(`x',6);} \\ 
& \texttt{vx = monomials(x,4:4);} \\ 
& \texttt{randn(`state',0)} \\ 
& \texttt{cp = randn(1,length(vx));} \\ 
& \texttt{p = cp*vx;} 
\end{flalign*}
This block of code constructs the polynomial that is to be lower bounded. In order to do this, we create a six dimensional vector \texttt{x} of variables using the \texttt{msspoly} command introduced before. Next, we construct a vector of monomials using the \texttt{monomials} function. The first input to this function is the \texttt{msspoly} variable over which the monomials are defined. The second input to the function is the range of degrees the monomials should have. In this case, since we are considering homogeneous quartics, the range is simply $4:4$ (i.e., just $4$). 
\begin{flalign*}
& \texttt{\% Build program} \\ 
& \texttt{prog = spotsosprog;} \\ 
& \texttt{prog = prog.withIndeterminate(x);} \\ 
& \texttt{[prog,gamma] = prog.newFree(1);} \\
& \texttt{prog = prog.withDSOS(p - gamma*(x'*x)\string^2);}
\end{flalign*}

This block of code sets up the program and constraints. First, the program is initialized in the form of the variable \texttt{prog}. This object will contain information about constraints and decision variables. Next, we declare the variable \texttt{x} to be an \emph{indeterminate} or abstract variable (thus distinguishing it from a decision variable). Decision variables are created using the \texttt{newFree} function in the class \texttt{spotsosprog}. The input to this function is the number of new decision variables to be created. The outputs are the updated program and a variable corresponding to the decision variable. In our case, we need a single variable \texttt{gamma} to be declared. Finally, we specify a DSOS constraint on \texttt{p - gamma*(x'*x)\string^2} using the \texttt{withDSOS} command. There are corresponding functions \texttt{withSDSOS} and \texttt{withSOS} that can be used to setup SDSOS or SOS constraints.
\begin{flalign*}
& \texttt{\% Setup options and solve program} \\
& \texttt{options = spot\_sdp\_default\_options();} \\
& \texttt{\% Use just the interior point algorithm to clean up} \\
& \texttt{options.solveroptions.MSK\_IPAR\_BI\_CLEAN\_OPTIMIZER =  ...} \\
& \texttt{... `MSK\_OPTIMIZER\_INTPNT';} \\
& \texttt{\% Don't use basis identification} \\
& \texttt{options.solveroptions.MSK\_IPAR\_INTPNT\_BASIS = `MSK\_BI\_NEVER';}  \\
& \texttt{\% Display solver output} \\
& \texttt{options.verbose = 1;} \\
& \nonumber \\
& \texttt{sol = prog.minimize(-gamma, @spot\_mosek, options);} \\
&  \nonumber \\
& \texttt{\%Get value of gamma for optimal solution} \\
& \texttt{gamma\_optimal = double(sol.eval(gamma))} 
\end{flalign*}

Finally, in this block of code, we {\gh dictate} an options structure for the program. In particular, the field \texttt{options.solveroptions} contains options that are specific to the solver to be used (MOSEK in our case). The \texttt{minimize} command is used to specify the objective of the program (which must be linear in the decision variables), the function handle corresponding to the solver to be used, and the structure of options. Currently, MOSEK, Gurobi and SeDuMi are the supported solvers. So, for example, in order to use the Gurobi solver for a (S)DSOS program, we would specify the function handle \texttt{@spot\_gurobi}. 

The output of the \texttt{minimize} command is a solution structure, which contains diagnostic information and allows one to access the optimized decision variables. The last line of our example code demonstrates how to obtain the optimized variable \texttt{gamma} and convert it to a MATLAB \texttt{double} type.

\subsection{Additional functionality}

Next, we review additional functionality not covered in the example above. While this section is not meant to be an exhaustive list of all the functionality available in SPOT, it should be sufficient to reproduce the numerical results presented in this paper. In the lines of code presented in the following sections, it is assumed that \texttt{prog} is an object of the \texttt{spotsosprog} class, \texttt{x} is a \texttt{msspoly} variable of size \texttt{n} and an indeterminate variable of \texttt{prog}. These variables can be initialized as in the example above with the following lines of code:
\begin{flalign}
& \texttt{prog = spotsosprog;} \nonumber \\
& \texttt{x = msspoly(`x',n);} \nonumber \\
& \texttt{prog = prog.withIndeterminate(x);} \nonumber 
\end{flalign}

\subsubsection{Creating decision variables} It is often useful to construct a polynomial whose coefficients are decision variables. This can be achieved with the \texttt{newFree} and \texttt{monomials} functions introduced above. In particular, one can create a vector of monomials in a \texttt{msspoly} variable \texttt{x} of a given degree \texttt{d} using the command \texttt{v = monomials(x,0:d)}. Then, the set of coefficients can be declared using \texttt{[prog,c] = prog.newFree(length(v))}. Finally, one can obtain the desired polynomial by multiplying these two together: \texttt{p = c'*v}. 

It is also possible to create other types of decision variables. For example, the following functions can be used to create various types of matrix decision variables:

\texttt{newSym}: New symmetric matrix.

\texttt{newDD}: New symmetric matrix constrained to be diagonally dominant.

\texttt{newSDD}: New symmetric matrix constrained to be scaled diagonally dominant.

\texttt{newPSD}: New symmetric positive semidefinite matrix.

\texttt{newDDdual}: New symmetric matrix constrained to lie in the dual of the cone of diagonally dominant matrices.

\texttt{newSDDdual}: New symmetric matrix constrained to lie in the dual of the cone of scaled diagonally dominant matrices.

The input to these functions is the size of the desired (square) matrix. For example, in order to create a $10 \times 10$ symmetric matrix \texttt{Q} constrained to be diagonally dominant, one can use the following lines of code:
\begin{flalign}
& \texttt{[prog,Q] = prog.newDD(10);} \nonumber
\end{flalign}

\subsubsection{Specifying constraints}

In addition to the \texttt{withDSOS}, \texttt{withSDSOS} and \texttt{withSOS} functions introduced previously, constraints on existing decision variables can be specified using the following functions:

\texttt{withEqs}: Sets up constraints of the form \texttt{expr = 0}, where \texttt{expr} is the input to the function and is a matrix whose elements are to be constrained to be equal to 0. 

\texttt{withPos}: Sets up constraints of the form \texttt{expr $\geq$ 0}, where \texttt{expr} is the input to the function and is a matrix whose elements are to be constrained to be nonnegative. 

Note that the functions \texttt{withEqs} and \texttt{withPos} allow one to specify constraints in a vectorized manner, thus avoiding MATLAB's potentially-slow for-loops. 

\texttt{withDD}: Constrains a symmetric matrix (the input to the function) to be diagonally dominant.

\texttt{withSDD}: Constrains a symmetric matrix (the input to the function) to be scaled diagonally dominant.

\texttt{withPSD}: Constrains a symmetric matrix (the input to the function) to be positive semidefinite.

In each case above, the inputs must be affine functions of the decision variables of the program. The output of these functions is an updated \texttt{spotsosprog} object that contains the new constraints. We illustrate the use of the function \texttt{withEqs} with the help of a simple example. The other functions can be used in a similar manner.
\begin{flalign}
& \texttt{prog = spotsosprog;} \nonumber \\
& \texttt{[prog,tau1] = prog.newFree(10);} \nonumber \\
& \texttt{[prog,tau2] = prog.newFree(10);} \nonumber \\
& \texttt{prog = prog.withEqs(tau1 - tau2);} 
\end{flalign}

In this example, the elements of the $10 \times 1$ decision vector \texttt{tau1} are constrained to be equal to the elements of the vector \texttt{tau2}. 

\texttt{withPolyEqs}: In order to constrain two polynomials $p_1(x)$ and $p_2(x)$ to be equal to each other for all $x$, one may use the function \texttt{withPolyEqs}. This function will constrain the coefficients of the two polynomials to be equal. Here is a simple example:
\begin{flalign}
& \texttt{prog = spotsosprog;} \nonumber \\
& \texttt{x = msspoly(`x',2);} \nonumber \\
& \texttt{prog = prog.withIndeterminate(x);} \nonumber \\
& \texttt{[prog,c] = prog.newFree(4);} \nonumber \\
& \texttt{p1 = (c(1) + c(2))*x(1)\string^2 + c(3)*x(2)\string^2;} \nonumber \\
& \texttt{p2 = 2*x(1)\string^2 + c(4)*x(2)\string^2;} \nonumber \\
& \texttt{prog = prog.withPolyEqs(p1 - p2);} \nonumber
\end{flalign}

This will constrain \texttt{c(1) + c(2) = 2} and \texttt{c(3) = c(4)}.

\subsubsection{Checking if a polynomial is dsos}

The three utility functions \texttt{isDSOS}, \texttt{isSDSOS} and \texttt{isSOS} allow one to check if a given polynomial is dsos, sdsos, or sos respectively. The only input to the functions is the polynomial to be checked. The outputs are a boolean variable indicating whether the polynomial is in fact dsos/ sdsos/ sos, the Gram matrix, and the monomial basis corresponding to the Gram matrix (these last two are non-empty only if the given polynomial is in fact dsos/sdsos/sos). The following lines of code provide a simple example:
\begin{flalign}
& \texttt{x = msspoly(`x',3);} \nonumber \\
& \texttt{p = x(1)\string^2 + 5*x(2)\string^2 + 3*x(3)\string^2;} \nonumber \\
& \texttt{[isdsos,Q,v] = isDSOS(p)}
\end{flalign}

One can check that the polynomial \texttt{p - v'*Q*v} is the zero polynomial (up to numerical tolerances).

\newpage

\bibliographystyle{siamplain}
\bibliography{pablo_amirali,elib}

\end{document}